\documentclass[a4paper,leqno]{amsart}

\usepackage{latexsym}
\usepackage[english]{babel}
\usepackage{fancyhdr}
\usepackage[mathscr]{eucal}
\usepackage{amsmath}
\usepackage{mathrsfs}
\usepackage{amsthm}
\usepackage{amsfonts}
\usepackage{amssymb}
\usepackage{amscd}
\usepackage{bbm}
\usepackage{graphicx}
\usepackage{graphics}
\usepackage{latexsym}
\usepackage{color}

\newcommand{\ud}{\mathrm{d}}

\newcommand{\ii}{\mathrm{i}}

\newcommand{\R}{\mathbb R}

\theoremstyle{plain}
\newtheorem{theorem}{Theorem}[section]
\newtheorem{lemma}[theorem]{Lemma}

\newtheorem{proposition}[theorem]{Proposition}

\theoremstyle{definition}

\newtheorem{remark}[theorem]{Remark}
\newtheorem*{remark*}{Remark}

\numberwithin{equation}{section}

\begin{document}

\title[Point-like perturbed fractional Laplacians with shrinking potentials]
{Point-like perturbed fractional Laplacians through shrinking potentials of finite range}
\author[A.~Michelangeli]{Alessandro Michelangeli}
\address[A.~Michelangeli]{International School for Advanced Studies -- SISSA \\ via Bonomea 265 \\ 34136 Trieste (Italy).}
\email{alemiche@sissa.it}
\author[R.~Scandone]{Raffaele Scandone}
\address[R.~Scandone]{International School for Advanced Studies -- SISSA \\ via Bonomea 265 \\ 34136 Trieste (Italy).}
\email{rscandone@sissa.it}

%\dedicatory{}

\begin{abstract}
We construct the rank-one, singular (point-like) perturbations of the $d$-dimensional fractional Laplacian in the physically meaningful norm-resolvent limit of fractional Schr\"{o}dinger operators with regular potentials centred around the perturbation point and shrinking to a delta-like shape. We analyse both possible regimes, the resonance-driven and the resonance-in\-de\-pend\-ent limit, depending on the power of the fractional Laplacian and the spatial dimension. To this aim, we also qualify the notion of zero-energy resonance for Schr\"{o}dinger operators formed by a fractional Laplacian and a regular potential.
\end{abstract}

\date{\today}

\subjclass[2000]{}

% 
% \thanks{Partially supported by the 2014-2017 MIUR-FIR 
% grant ``\emph{Cond-Math: Condensed Matter and Mathematical Physics}'' code RBFR13WAET}

\maketitle

%\tableofcontents

\section{Introduction and background}\label{intro}

In the last decade an amount of studies focused, in particular in application to the context of fractional quantum mechanics, on linear Schr\"{o}dinger equations governed by the linear operator
\begin{equation}\label{eq:simbolicOP}
 (-\Delta)^{s/2}\;+\;\textrm{singular perturbation at $x_0$}
\end{equation}
for some fixed point $x_0\in\mathbb{R}^d$ and some $s>0$, that is, Schr\"{o}dinger equations for a singular perturbation of a fractional power of the Laplacian \cite{Muslih-IntJTP-2010-3D-fracLaplandDelta,COliveira-Costa-Vaz-JMP2010,COliveira-Vaz-JPA2011_1D_fracLaplandDelta,Lenzi-etAl-JMP2013_fracLapl_andDelta,Sandev-Petreska-Lenzi-JMP2014,Tare-Esguerra-JMP2014,Jarosz-Vaz-JMP2016_1D_gs_fracLaplandDelta,Nayga-Esguerra-JMP-2016,Sacchetti-fractional2018}.

Motivated by that, in a recent work in collaboration with A.~Ottolini \cite{MOS-2018-fractional-and-perturbation} we set up the systematic construction and classification of all the self-adjoint realisations in $L^2(\mathbb{R}^d)$ of the operators of the form \eqref{eq:simbolicOP} through a natural `restriction-extension' procedure: first one restricts the operator $(-\Delta)^{s/2}$ (initially defined, e.g., as a Fourier multiplier) to smooth functions vanishing in neighbourhoods of $x_0$, and then one builds all the operator extensions of such restriction that are self-adjoint on $L^2(\mathbb{R}^d)$.

This approach is surely satisfactory from the point of view of the interpretation of the output operator, which by construction is to be regarded as a point-like perturbation of the fractional Laplacian through an interaction supported only at $x_0$, say, ``$(-\Delta)^{s/2}+\delta(x-x_0)$''. However, it obfuscates an amount of physical meaning, since it does not provide information, as the intuition would make one expect instead, on how the actual singular perturbation \eqref{eq:simbolicOP} is approximatively realised as a genuine pseudo-differential operator $(-\Delta)^{s/2}+V(x-x_0)$ with a regular potential $V$ centred around $x=0$, with sufficiently short range and strong magnitude.

For the non-fractional Laplacian $-\Delta$ in $L^2(\mathbb{R}^d)$, the realisation of a singular perturbation at $x_0\in\mathbb{R}^d$ by means of approximating Schr\"{o}dinger operators $-\Delta+V_\varepsilon$ with regular potentials $V_\varepsilon$ spiking up and shrinking around $x_0$ at a spatial scale $\varepsilon^{-1}$ in the limit $\varepsilon\downarrow 0$ is known since long for dimension $d=1$ \cite{aghkk-1984-1D-point-int}
, $d=2$ \cite{AHKH-1987-2D}, and $d=3$ \cite{AGHK-1982} (we also refer to \cite{albeverio-solvable,albverio-kurasov-2000-sing_pert_diff_ops} for a comprehensive overview), that is, all the dimensions in which non-trivial singular perturbations exist.

The analogous question for the fractional Laplacian $(-\Delta)^{s/2}$ was unanswered so far, and we solve it in the present work.

Not only is it topical in view of the above-mentioned recent mainstream in the literature of fractional Schr\"{o}dinger equations with singular perturbation, but also it rises up the conceptually new issue of how a \emph{local} potential $V_\varepsilon$ can be suitably re-scaled so as to produce the desired perturbation of the \emph{non-local} operator $(-\Delta)^{s/2}$.

Let us first of all reconsider what emerges from the construction that, as mentioned, was recently given in \cite{MOS-2018-fractional-and-perturbation}.

For $s>0$ and $d\in\mathbb{N}$, the restriction $(-\Delta)^{s/2}|_{C^\infty_0(\mathbb{R}^d\setminus \{0\})}$ is a positive symmetric operator on the Hilbert space $L^2(\mathbb{R}^d)$, hence with equal deficiency indices, and for short we shall just speak of \emph{the} deficiency index. The number
\[
 \mathcal{J}(s,d)\;:=\;\textrm{deficiency index of $(-\Delta)^{s/2}|_{C^\infty_0(\mathbb{R}^d\setminus \{0\})}$}
\]
is \emph{finite}, and is zero or a strictly positive integer depending on $s$ and $d$,
% Thus, for concreteness, in dimension $d=3$ one has
% \begin{equation}
%  s\,\in\,I_n\;:=\;
%  \begin{cases}
%   \;\qquad(0,\frac{3}{2}] & n=0 \\
%   \;(\frac{2n+1}{2},\frac{2n+3}{2}] & n =1,2,\dots
%  \end{cases}\qquad \Rightarrow \qquad 
%  \mathcal{J}(s,3)\;=\;
%  \begin{pmatrix}
%   n + 2 \\ 3
%  \end{pmatrix}
% \end{equation}
according to the rule
\begin{equation}\label{eq:rule}
\begin{split}
 s\in I_{n}^{(d)}\,&:=\,
 \begin{cases}
  \qquad(0,\frac{d}{2}] & n=0 \\
  (\frac{d}{2}+n-1,\frac{d}{2}+n] & n =1,2,\dots
 \end{cases} \\
 &\;\Rightarrow \quad 
 \mathcal{J}(s,d)\;=\;
 \textstyle\begin{pmatrix}
  d+n-1 \\ d
 \end{pmatrix}.
\end{split}
\end{equation}
%that for the reader's convenience we reproduce in the appendix.  (\textcolor{red}{MA ANCHE NO, VEDIAMO}).

In the non-fractional case $s=2$ this yields the familiar values $\mathcal{J}(2,1)=2$, $\mathcal{J}(2,2)=1$, $\mathcal{J}(2,3)=1$, and $\mathcal{J}(2,d)=0$ for $d\geqslant 4$.

As well known, $\mathcal{J}(s,d)$ quantifies the infinite multiplicity of self-adjoint extensions of $(-\Delta)^{s/2}|_{C^\infty_0(\mathbb{R}^d\setminus \{0\})}$ in $L^2(\mathbb{R}^d)$. By means of standard methods of the Kre{\u\i}n-Vi\v{s}ik-Birman theory \cite{GMO-KVB2017} one sees that the domain of each extension is formed by functions that are canonically decomposed into a regular $H^s$-component and a more singular component, the latter belonging to the $\mathcal{J}(s,d)$-dimensional kernel of $((-\Delta)^{s/2}|_{C^\infty_0(\mathbb{R}^d\setminus \{0\})})^*+\lambda\mathbbm{1}$ for some arbitrarily chosen $\lambda>0$, and satisfy an amount of `boundary' (or `contact') conditions between the evaluation at $x=0$ of the regular part or of some if its derivatives and the coefficients of the leading singularities of the singular part as $x\to 0$. Each set of boundary conditions identifies uniquely an extension.

For concreteness of the presentation, in this work we consider the self-adjoint extensions of $(-\Delta)^{s/2}|_{C^\infty_0(\mathbb{R}^d\setminus \{0\})}$ in the case of deficiency index 1 only, and for simplicity we omit further the explicit discussion of the `endpoint' values of $s$, namely the largest possible value, at given $d$, compatible with $\mathcal{J}(s,d)=1$. As expressed by \eqref{eq:rule}, this amounts to analysing the regime $s\in(\frac{1}{2},\frac{3}{2})$ in $d=1$, $s\in(1,2)$ in $d=2$, $s\in(\frac{3}{2},\frac{5}{2})$ in $d=3$, etc., where the considered intervals are the non-endpoint values of $s$, the endpoint value being $s=\frac{d}{2}+1$. We shall refer to such cases as the `$\mathcal{J}=1$ scenario'. For this scenario we then discuss how to realise the corresponding extensions in the limit of Schr\"{o}dinger operators with fractional Laplacian and shrinking potentials, say, $(-\Delta)^{s/2}+V_\varepsilon$ as $\varepsilon\downarrow 0$.

In fact, it will be evident from our discussion that the behaviour and the control of the limit $\varepsilon\downarrow 0$ in the $\mathcal{J}=1$ scenario is technically the very same irrespectively of the dimension, and therefore we will pick up a concrete value of $d$ for the explicit computations, \emph{modulo the dichotomy that we now describe}.

When $\mathcal{J}(s,d)=1$, and $s\neq\frac{d}{2}+1$, the space where the above-mentioned singular components run over, namely $\ker(((-\Delta)^{s/2}|_{C^\infty_0(\mathbb{R}^d\setminus\{0\})})^*+\lambda\mathbbm{1})$, is the one-dimensional space spanned by the Green function $\mathsf{G}_{s,\lambda}$ of the fractional Laplacian, defined by
\[
 ((-\Delta)^{s/2}+\lambda)\mathsf{G}_{s,\lambda}\;=\;\delta(x)\,;
\]
for higher deficiency index the kernel is spanned by $\mathsf{G}_{s,\lambda}$ \emph{and} other non-$H^s$ functions. Now, depending on $d$ and $s$, the Green function $\mathsf{G}_{s,0}$ may be singular or regular at $x=0$: when $s<d$, $\mathsf{G}_{s,0}$ has a singularity $\sim|x|^{-(d-s)}$, it has a logarithmic singularity when $s=d$, and it is continuous at $x=0$ when $s>d$. Omitting the transition case, which does not alter the conceptual scheme of the present discussion and could be easily recovered with analogous arguments to those that we shall use when $s<d$, we thus distinguish two possibilities in the $\mathcal{J}=1$ scenario, that we call
\begin{itemize}
 \item \textsc{locally singular, or resonance-driven case:} $s<d$,
 \item \textsc{locally regular, or resonance-independent case:} $s>d$.
\end{itemize}
We shall explain in a moment the meaning of the `resonance' jargon: it has to do with how the limit of shrinking potentials must be organised in order to reach a self-adjoint extension of $(-\Delta)^{s/2}|_{C^\infty_0(\mathbb{R}^d\setminus \{0\})}$ in one case or in the other. Also, let us remark that an analogous dichotomy occurs when the deficiency index of $(-\Delta)^{s/2}|_{C^\infty_0(\mathbb{R}^d\setminus \{0\})}$ is larger than 1: the singular (non-$H^s$) component of the elements in the domain of the considered self-adjoint extension may or may not display a local singularity as $x\to 0$.

In view of the above alternative, we make the following presentational choice. Since in all dimensions $d$ but $d=1$ the interval $s\in(\frac{d}{2},\frac{d}{2}+1)$ corresponding to deficiency index 1 lies strictly below the transition value $s=d$ that separates the locally regular from the locally singular regime, as a representative of any such value of $d$ for concreteness we choose $d=3$: the discussion on the limit of shrinking potentials would then be \emph{immediately exportable to any other $d\geqslant 2$}. Next to that, we also discuss the case $d=1$, where instead the interval $s\in(1,2)$ corresponding to deficiency index 1 contains the transition value $s=1$.

As mentioned above, self-adjoint extensions of $(-\Delta)^{s/2}|_{C^\infty_0(\mathbb{R}^d\setminus \{0\})}$ in the locally regular and the locally singular case differ both in the type of non-regularity of the functions in their domain at $x=0$, and, as we shall show in this work, in the type of approximating Schr\"{o}dinger operators $(-\Delta)^{s/2}+V_\varepsilon$, meaning, in the scaling chosen for $V_\varepsilon$ and, most importantly, in the spectral requirements.

Extensions in the locally regular case can be reached as $\varepsilon\downarrow 0$ through suitably rescaled versions $V_\varepsilon$ of a given potential $V$ with no further prescription on $V$ but those technical assumptions ensuring that the limit itself is well-posed. Instead, extensions in the locally singular case can only be reached if the unscaled operator $(-\Delta)^{s/2}+V$ admits a \emph{zero-energy resonance}, a spectral behaviour at the bottom of its essential spectrum which we shall define in due time and roughly speaking amounts to the existence of a suitably decaying, non square-integrable, $L^2_\mathrm{loc}$-solution $f$ to $((-\Delta)^{s/2}+V)f=0$. In a sense that we shall make precise, this difference is due to the fact that a zero-energy resonance is needed in the approximating fractional Schr\"{o}dinger operator in order to reproduce in the limit the locally singular behaviour in the domain of the considered self-adjoint extension.

In fact, the phenomenon we have just described is the generalisation for $(-\Delta)^{s/2}$ of what is well known for $-\Delta$ (i.e., $s=2$ in the present notation). When $d=1$, the deficiency index of $(-\Delta)|_{C^\infty_0(\mathbb{R}\setminus \{0\})}$ equals 2 and
\[
 \ker\big(((-\Delta)|_{C^\infty_0(\mathbb{R}\setminus \{0\})})^*+\lambda\mathbbm{1}\big)\;=\;\mathrm{span}\big\{{\textstyle \frac{1}{\sqrt{\lambda}}\,e^{-\sqrt{\lambda}|x|}}, (\mathrm{sign} \,x)\,e^{-\sqrt{\lambda}|x|}\big\}
\]
(in particular, $\mathsf{G}_{2,\lambda}(x)=\sqrt{\frac{\pi}{2\lambda}}\,e^{-\sqrt{\lambda}|x|}$), therefore the functions in the above space are less regular than $H^2(\mathbb{R})$ but not locally singular at $x=0$. The so-called $\delta$-type extensions, namely those in the domain of which the singular component is $e^{-\sqrt{\lambda}|x|}$, can indeed be realised as limits of $-\Delta+V_\varepsilon$ with no spectral requirement needed at energy zero for the unscaled $-\Delta+V$ \cite[Chapt.~I.3]{albeverio-solvable}. On the contrary, when $d=3$ the deficiency index of $(-\Delta)|_{C^\infty_0(\mathbb{R}^3\setminus \{0\})}$ equals 1 and
\[
 \ker\big(((-\Delta)|_{C^\infty_0(\mathbb{R}^3\setminus \{0\})})^*+\lambda\mathbbm{1}\big)\;=\;\mathrm{span}\{\mathsf{G}_{2,\lambda}\}\,,\qquad  \mathsf{G}_{2,\lambda}(x)=\frac{e^{-\sqrt{\lambda}|x|}}{4\pi|x|}\,,
\]
thus with a local singularity at $x=0$. The self-adjoint extensions of $(-\Delta)|_{C^\infty_0(\mathbb{R}^3\setminus \{0\})}$ can be realised as limits of $-\Delta+V_\varepsilon$ provided that $-\Delta+V$ is zero-energy resonant \cite[Chapt.~I.1]{albeverio-solvable}. In the former situation we are in the locally regular, resonant-independent case; in the latter we are in the locally singular, resonant-driven case.

The material of this work is organised as follows.
\begin{itemize}
 \item In Section \ref{sec:3D} we define the singular perturbations of the three-dimensional fractional Laplacian and we present the approximation scheme in terms of fractional Schr\"odinger operators with regular, shrinking potentials.
 \item In Section \ref{sec:1D} we present the one-dimensional analogue, including the definition of the singular perturbations and the two distinct approximation schemes, for the resonance-driven and the resonance-independent cases.
 \item Section \ref{sec:proof_approx_3D} contains the proof of the three-dimensional limit.
 \item Section \ref{sec:proof_approx_1D} contains the proof of the one-dimensional limit in the resonance-driven case. From the technical point of view, the argument here is completely analogous to that of \ref{sec:proof_approx_3D}, as the 3D case too is resonance-driven.
 \item Section \ref{sec:proof_approx_1D-nores} contains instead the proof of the one-dimensional limit in the resonance-independent case.
 \item In Section \ref{sec:zero_en_res} we prove a technical result used in the main proofs, that is, the characterisation of the zero-energy resonant behaviour of the unscaled operator $(-\Delta)^{s/2}+V$. Then, we discuss the occurrence of zero-energy resonances.
% \item Last, in Appendix \ref{app:1D_sing_pert} for the benefit of the reader we
\end{itemize}

Let us conclude this Introduction with a few comments about our otherwise standard notation. For an operator $T$ on a Hilbert space, $\mathcal{D}(T)$ denotes its operator domain and, when $T$ is self-adjoint, $\mathcal{D}[T]$ denotes its form domain. We shall denote by $\mathbbm{1}$, resp., by $\mathbbm{O}$, the identity and the null operator on any of the considered Hilbert spaces. We shall indicate the Fourier transform by $\widehat{\phi}$ or $\mathcal{F}\phi$ with the convention $\widehat{\phi}(p)=(2\pi)^{-\frac{d}{2}}\int_{\mathbb{R}^d}e^{-\ii p x}\phi(x)\ud x $. We shall write $A\lesssim B$ for $A\leqslant\mathrm{const.}\,B$ when the constant does not depend on the other relevant parameters or variables of both sides of the inequality; for $x\in\mathbb{R}^d$ we shall write $\langle x\rangle:=(1+x^2)^{\frac{1}{2}}$.

%When $s=2$ and $d=1$, and hence the deficiency index equals 2, the family of $\infty^4$ self-adjoint realisations is well understood in the literature, among which the so-called $\delta$-extensions and $\delta'$-extensions \cite[Chapt.~I.3 and I.4]{albeverio-solvable}.

\section{Approximation scheme in dimension three}\label{sec:3D}

In this Section we consider the singular perturbations of the three-dimensional fractional Laplacian and their approximation by means of fractional Schr\"odinger operators with shrinking potentials.

Let us start with the densely defined, closed, positive, symmetric operator 
\begin{equation}
 \mathring{\mathsf{k}}^{(s/2)}\;:=\;\overline{(-\Delta)^{s/2}\upharpoonright C^\infty_0(\mathbb{R}^3\setminus \{0\})}\,,\qquad s>0\,,
\end{equation}
with respect to the Hilbert space $L^2(\mathbb{R}^3)$. In \cite{MOS-2018-fractional-and-perturbation} we presented the construction and classification of the self-adjoint extensions of $\mathring{\mathsf{k}}^{(s/2)}$, which we recall here below.

Clearly, for small enough powers $s$, $\mathring{\mathsf{k}}^{(s/2)}$ is already self-adjoint, thus with no room for point-like singular perturbations, indeed \cite[Lemma A.1]{MOS-2018-fractional-and-perturbation}
\begin{equation}
 \mathcal{D}(\mathring{\mathsf{k}}^{(s/2)})\;=\;H^s_0(\mathbb{R}^3\setminus \{0\})\;=\;\overline{C^\infty_0(\mathbb{R}^3\setminus \{0\})\,}^{\|\,\|_{H^s}}=\;H^s(\mathbb{R}^3)\,,\quad \textrm{if } s\in[0,{\textstyle\frac{3}{2}})\,. \!\!\!
\end{equation}

When $s$ increases, the domain of $\mathring{\mathsf{k}}^{(s/2)}$ is qualified by an increasing number of conditions for $H^s$-functions, namely \cite[Lemma A.2]{MOS-2018-fractional-and-perturbation}
\begin{equation}
 \begin{split}
  \mathcal{D}(\mathring{\mathsf{k}}^{(s/2)})\;&=\;\left\{\!\!
 \begin{array}{c}
  f\in H^s(\mathbb{R}^3) \textrm{ such that}\\
  \int_{\mathbb{R}^3}p_1^{\gamma_1}p_2^{\gamma_2}p_3^{\gamma_3}\widehat{f}(p)\,\ud p=0 \\
  \gamma_1,\gamma_2,\gamma_3\in\mathbb{N}_0\,,\;\gamma_1+\gamma_2+\gamma_3\leqslant n-1
 \end{array}\!\!
 \right\} \\
 \textrm{if } s\;&\in\;I_n^{(3)}\;:=\;\textstyle(n+\frac{1}{2},n+\frac{3}{2})\,,\qquad n\in\,\mathbb{N}\,;
 \end{split}
\end{equation}
as a consequence, the adjoint of $\mathring{\mathsf{k}}^{(s/2)}$ becomes strictly larger than $\mathring{\mathsf{k}}^{(s/2)}$ itself, with an increasingly complicated structure of its domain that reflects the fact that for $s\in I_n^{(3)}$ the deficiency index of $\mathring{\mathsf{k}}^{(s/2)}$ equals $\begin{pmatrix} n+2 \\ 3 \end{pmatrix}$, and this in turn affects the structure of the family of its self-adjoint extensions.

In particular, in the regime $s\in I_1^{(3)}=(\frac{3}{2},\frac{5}{2})$ one has 
\begin{equation}
 \mathcal{D}(\mathring{\mathsf{k}}^{(s/2)})\;=\;\Big\{f\in H^s(\mathbb{R}^3)\,\Big|\!\int_{\mathbb{R}^3}\widehat{f}(p)\,\ud p=0\Big\} \qquad\textrm{if } s\in\textstyle(\frac{3}{2},\frac{5}{2})
\end{equation}
and $\mathring{\mathsf{k}}^{(s/2)}$ has deficiency index 1, which leaves room for a one-(real-)parameter family of self-adjoint extensions. In order to qualify them, for chosen $\lambda>0$ and  $s\in\mathbb{R}$ let us denote the Green's function as
\begin{equation}\label{eq:sfG}
  \mathsf{G}_{s,\lambda}(x)\;:=\;\frac{1}{\:(2\pi)^{\frac{3}{2}}}\Big(\frac{1}{|p|^{s}+\lambda}\Big)^{\!\vee}(x)\,,\qquad x,p\in\mathbb{R}^3\,.
\end{equation}
By construction, distributionally.
\begin{equation}
 ((-\Delta)^{s/2}+\lambda) \,\mathsf{G}_{s,\lambda}\;=\;\delta(x)\,.
\end{equation}
Observe that $\mathsf{G}_{s,\lambda}$ has a local singularity $|x|^{-(3-s)}$, more precisely \cite[Sec.~3]{MOS-2018-fractional-and-perturbation},
\begin{equation}\label{eq:sfGlambda_asympt}
 \mathsf{G}_{s,\lambda}(x)\;=\;\frac{\Lambda_s}{\;|x|^{(3-s)}}+\mathsf{J}_{s,\lambda}(x) 
\end{equation}
with
\begin{equation}\label{eq:sfGlambda_asympt-2}
 \begin{split}
  \Lambda_s\;&:=\;\frac{\Gamma(\frac{3-s}{2})}{\,(2\pi)^{\frac{3}{2}}\,2^{s-\frac{3}{2}}\,\Gamma(\frac{s}{2})} \\
  \mathsf{J}_{s,\lambda}\;&:=\;-\frac{\lambda}{\:(2\pi)^{\frac{3}{2}}}\Big(\frac{1}{|p|^s(|p|^{s}+\lambda)}\Big)^{\!\vee}\;\in\;C_\infty(\mathbb{R}^3)\,.
 \end{split}
\end{equation}

The following construction/classification Theorem was established in \cite{MOS-2018-fractional-and-perturbation}.

\begin{theorem}\label{thm:Kalpha} Let $s\in(\frac{3}{2},\frac{5}{2})$.

\begin{itemize}
 \item[(i)] The self-adjoint extensions in $L^2(\mathbb{R}^3)$ of the operator $\mathring{\mathsf{k}}^{(s/2)}$ form the family $(\mathsf{k}^{(s/2)}_\alpha)_{\alpha\in\mathbb{R}\cup\{\infty\}}$, where $\mathsf{k}^{(s/2)}_\infty$ is its Friedrichs extension, namely the self-adjoint fractional Laplacian $(-\Delta)^{s/2}$, and all other (proper) extensions are given, for arbitrary $\lambda>0$, by
  \begin{equation}\label{eq:domKwithalpha}
 \begin{split}
  \mathcal{D}(\mathsf{k}^{(s/2)}_\alpha)\;&=\;\left\{\left. g=F^\lambda+{\displaystyle\frac{F^\lambda(0)}{\,\alpha-\frac{\lambda^{\frac{3}{s}-1}}{2\pi s\sin(\frac{3\pi}{s})}}\,}\,\mathsf{G}_{s,\lambda}\,\right|F^\lambda\in H^s(\mathbb{R}^3)\right\} \\
  (\mathsf{k}^{(s/2)}_\alpha+\lambda)\,g\;&=\;((-\Delta)^{s/2}+\lambda)\,F^\lambda\,.
 \end{split}
\end{equation}
  \item[(ii)] For  each $\alpha\in\mathbb{R}$ the quadratic form of the extension $\mathsf{k}^{(s/2)}_\alpha$ is given by
\begin{eqnarray}
   \mathcal{D}[\mathsf{k}^{(s/2)}_\alpha]\!&=&\!\! H^{\frac{s}{2}}(\mathbb{R}^3)\dotplus \mathrm{span}\{\mathsf{G}_{s,\lambda}\}  \label{eq:Kring-alpha_form1}\\
 \qquad \mathsf{k}^{(s/2)}_\alpha[F^\lambda+\kappa_\lambda \mathsf{G}_{s,\lambda}]\!\!&=&\!\! \||\nabla|^s F^\lambda\|_{L^2(\mathbb{R}^3)}^2-\lambda\|F^\lambda+\kappa_\lambda \mathsf{G}_{s,\lambda}\|_{L^2(\mathbb{R}^3)}^2 \nonumber \\
 & & \!\!\!\!\!\!\!\!\!\!\!\!\!\!\!\!\!\!\!\!\!\!\!\!+\,\lambda\|F^\lambda\|_{L^2(\mathbb{R}^3)}^2  +{\textstyle\big(\alpha-\frac{\lambda^{\frac{3}{s}-1}}{2\pi s\sin(\frac{3\pi}{s})}\big)}|\kappa_\lambda|^2 \label{eq:Kring-alpha_form2}
 \end{eqnarray}
 for arbitrary $\lambda>0$.
  \item[(iii)] The resolvent of $\mathsf{k}^{(s/2)}_\alpha$ is given by
  \begin{equation}\label{eq:Kalpha_res}
   \begin{split}
      (\mathsf{k}^{(s/2)}_\alpha+\lambda\mathbbm{1})^{-1}\;&=\;((-\Delta)^{s/2}+\lambda\mathbbm{1})^{-1} \\
      &\quad +{\textstyle\big(\alpha-\frac{\lambda^{\frac{3}{s}-1}}{2\pi s\sin(\frac{3\pi}{s})}\big)^{\!-1}}\, |\mathsf{G}_{s,\lambda}\rangle\langle \mathsf{G}_{s,\lambda}|
   \end{split}
  \end{equation}
  for arbitrary $\lambda>0$.
 \item[(iv)] Each extension is semi-bounded from below, and
 \begin{equation}\label{eq:spec-kalpha}
  \begin{split}
   \sigma_{\mathrm{ess}}(\mathsf{k}^{(s/2)}_\alpha)\;&=\;\sigma_{\mathrm{ac}}(\mathsf{k}^{(s/2)})\;=  \;[0,+\infty)\,,\qquad \sigma_{\mathrm{sc}}(\mathsf{k}^{(s/2)})\;=\;\emptyset\,, \\
   \sigma_\mathrm{disc}(\mathsf{k}^{(s/2)}_\alpha)\;&=\;
   \begin{cases}
    \quad \emptyset & \textrm{ if } \alpha\geqslant 0 \\
    \{E_\alpha^{(s)}\} & \textrm{ if } \alpha < 0\,,
   \end{cases}
  \end{split}
 \end{equation}
 where the eigenvalue $E_\alpha^{(s)}$ is non-degenerate and is given by
 \begin{equation}\label{eq:KalphanegEV}
 E_\alpha^{(s)}\;=\;-\big(2\pi|\alpha|\,s\,\sin(-{\textstyle\frac{3\pi}{s})}\big)^{\frac{s}{3-s}}\,,
\end{equation}
  the (non-normalised) eigenfunction being $\mathsf{G}_{s,\lambda=|E_\alpha^{(s)}|}$.
 \end{itemize}
\end{theorem}

Our goal now is to qualify each of the extensions given by Theorem \ref{thm:Kalpha} as suitable limits of approximating fractional Schr\"{o}dinger operators with finite range potentials.

It is convenient to introduce the class $\mathcal{R}_{s,d}$, $d\in\mathbb{N}$, $s\in(\frac{d}{2},d)$, of measurable functions $V:\mathbb{R}^d\to\mathbb{C}$ such that
\begin{equation}\label{eq:Rclass}
 \iint_{\mathbb{R}^d\times\mathbb{R}^d}\!\!\ud x\,\ud y\,\frac{|V(x)|\,|V(y)|}{\;|x-y|^{2(d-s)}\,}\;=:\;\|V\|_{\mathcal{R}_{s,d}}^2\;<\;+\infty\,.
\end{equation}
$\mathcal{R}_{2,3}$ is the well-known Rollnick class on $\mathbb{R}^3$. Clearly, $\mathcal{R}_{s,d}\supset C^\infty_0(\mathbb{R}^d)$.

For each $s\in(\frac{3}{2},\frac{5}{2})$ we make the following assumption.

\medskip

\textbf{Assumption (I$_s$).}
\begin{itemize}
 \item[(i)] $V:\mathbb{R}^3\to\mathbb{R}$ is a measurable function in $L^1(\mathbb{R}^3,\langle x\rangle^{2s-3}\ud x)\cap\mathcal{R}_{s,3}$.
 \item[(ii)] $\eta:\mathbb{R}\to\mathbb{R}^+$ is a continuous function satisfying $\eta(0)=\eta(1)=1$ and
 \[
  \eta(\varepsilon)\;=\;1+\eta_s\,\varepsilon^{3-s}+o(\varepsilon^{3-s})\qquad\textrm{as }\varepsilon\downarrow 0
 \]
 for some $\eta_s\in\mathbb{R}$ that we call the \emph{strength} of the \emph{distortion factor} $\eta$.
\end{itemize}

\medskip

For given $V$ and $\eta$ satisfying Assumption (I$_s$), let us set 
\begin{equation}\label{eq:heps-Veps}
 h_\varepsilon^{(s/2)}\;:=\;(-\Delta)^{s/2}+V_\varepsilon\,,\qquad V_\varepsilon(x)\;:=\;\frac{\,\eta(\varepsilon)}{\:\varepsilon^s}\,V(\textstyle\frac{x}{\varepsilon})\,,\qquad\varepsilon>0\,.
\end{equation}
For every $\varepsilon>0$ the operator $h_\varepsilon^{(s/2)}$, defined as a form sum, is self-adjoint on $L^2(\mathbb{R}^3)$ and $\sigma_{\mathrm{ess}}(h_\varepsilon^{(s/2)})=[0,+\infty)$ (Lemma \ref{lem:compactness}(iii)).

The spectral properties of the unscaled operator $(-\Delta)^{s/2}+V$ at the bottom of the essential spectrum are crucial for the limit $\varepsilon\downarrow 0$ in $h_\varepsilon^{(s/2)}$. In the next Theorem we qualify the zero-energy behaviour of $(-\Delta)^{s/2}+V$.

\begin{theorem}\label{thm:resonance}
 Let $s\in(\frac{3}{2},\frac{5}{2})$, $V\in L^1(\mathbb{R}^3,\langle x\rangle^{2s-3}\ud x)\cap\mathcal{R}_{s,3}$, real-valued. Let $v:=|V|^{\frac{1}{2}}$ and $u:=|V|^{\frac{1}{2}}\mathrm{sign}(V)$.
 \begin{itemize}
  \item[(i)] The operator $u(-\Delta)^{-\frac{s}{2}}v$ is compact on $L^2(\mathbb{R}^3)$.
 \end{itemize}
 Assume in addition that 
 \begin{equation}\label{eq:phi}
  u(-\Delta)^{-\frac{s}{2}}v\,\phi\;=\;-\phi\qquad\textrm{for some } \phi\in L^2(\mathbb{R}^3)\setminus \{0\}
 \end{equation}
 and define
 \begin{equation}\label{eq:psi}
  \psi\;:=\;(-\Delta)^{-\frac{s}{2}}v\,\phi\,.
 \end{equation}
 Then:
 \begin{itemize}
  \item[(ii)] $\psi\in L^2_{\mathrm{loc}}(\mathbb{R}^3)$ and  $\big((-\Delta)^{s/2}+V\big)\psi=0$ in the sense of distributions,
  \item[(iii)] $\langle v,\phi\rangle_{L^2}\;=\;-\displaystyle\int_{\mathbb{R}^3}\ud x\,V(x)\psi(x)$,
  \item[(iv)] $\psi\in L^2(\mathbb{R}^3)$ $\Leftrightarrow$ $\langle v,\phi\rangle_{L^2}=0$, in which case $\psi\in\mathcal{D}((-\Delta)^{s/2}+V)$.
 \end{itemize}
\end{theorem}

When a $L^2$-function $\phi$ exists that satisfies \eqref{eq:phi} and the corresponding function $\psi$ defined by \eqref{eq:psi} belongs to $L^2_{\mathrm{loc}}(\mathbb{R}^3)\setminus L^2(\mathbb{R}^3)$ we say that $(-\Delta)^{s/2}+V$ is \emph{zero-energy resonant} and that $\psi$ is a \emph{zero-energy resonance} for $(-\Delta)^{s/2}+V$. If for the zero-energy resonant operator $(-\Delta)^{s/2}+V$ the eigenvalue $-1$ of $u(-\Delta)^{-\frac{s}{2}}v$ is non-degenerate, then we say that the resonance is \emph{simple}.
Of course, if $\psi\in L^2(\mathbb{R}^3)$, then $\psi$ is an eigenfunction of $(-\Delta)^{s/2}+V$ with eigenvalue zero.

We shall prove Theorem \ref{thm:resonance} in Section \ref{sec:zero_en_res} together with a discussion of the occurrence of a zero-energy resonance for $(-\Delta)^{s/2}+V$ (Proposition \ref{prop:exist_res}).

Let us now formulate our main result for dimension three. It is the control of the approximation, in the norm resolvent sense, of the singular perturbation operator $\mathsf{k}^{(s/2)}_\alpha$ by means of Schr\"{o}dinger operators with the $\frac{s}{2}$-th fractional Laplacian and shrinking potentials $V_\varepsilon$ around the origin. We shall prove it in Section \ref{sec:proof_approx_3D}.

\begin{theorem}\label{thm:shrinking3D}
 Let $s\in(\frac{3}{2},\frac{5}{2})$. Given a potential $V$ and a distortion factor $\eta$ with strength $\eta_s$ satisfying Assumption \emph{(I$_s$)}, for every $\varepsilon>0$ let $h_\varepsilon^{(s/2)}=(-\Delta)^{s/2}+V_\varepsilon$ be the corresponding self-adjoint Schr\"odinger operator defined in \eqref{eq:heps-Veps} with the $\frac{s}{2}$-th fractional Laplacian and the shrinking potential $V_\varepsilon$.
\begin{itemize}
 \item[(i)] If $(-\Delta)^{s/2}+V$ is not zero-energy resonant, then $h_\varepsilon^{(s/2)}\xrightarrow[]{\;\varepsilon\downarrow 0\;}(-\Delta)^{s/2}$ in the norm-resolvent sense on $L^2(\mathbb{R}^3)$.
 \item[(ii)] If $(-\Delta)^{s/2}+V$ admits a \emph{simple} zero-energy resonance $\psi$, then for  
 \[
  \alpha\;:=\;-\eta_s\, \Big|\!\int_{\mathbb{R}^3}\ud x\,V(x)\psi(x)\,\Big|^{-2}
 \]
 one has $h_\varepsilon^{(s/2)}\xrightarrow[]{\;\varepsilon\downarrow 0\;}\mathsf{k}^{(s/2)}_\alpha$ in the norm-resolvent sense on $L^2(\mathbb{R}^3)$.
\end{itemize}
\end{theorem}

In view of the discussion we made in the introductory Section, the two possible alternatives in Theorem \ref{thm:shrinking3D} are the manifestation of the locally singular, resonant-driven nature of the limit: the limit is well-posed for a generic class of potentials $V$, but it is non-trivial only if additionally $(-\Delta)^{s/2}+V$ is zero-energy resonant.

By a simple scaling argument one sees that $(-\Delta)^{s/2}+V_\varepsilon$ remains zero-energy resonant for any $\varepsilon>0$ if the scaling is `purely geometric', namely with trivial distortion factor, $\eta(\varepsilon)\equiv 1$. In this case, the signature of the resonance is particularly transparent: as stated in Theorem \ref{thm:shrinking3D}(ii), the limit $\varepsilon\downarrow 0$ with $\eta(\varepsilon)\equiv 1$ produces the extension parametrised by $\alpha=0$ and we see from Theorem \ref{thm:Kalpha}(iv) that the negative eigenvalue of $\mathsf{k}^{(s/2)}_\alpha$ when $\alpha<0$ converges to 0 as $\alpha\uparrow 0$, with the corresponding eigenfunction $\mathsf{G}_{s,\lambda=|E_\alpha^{(s)}|}$ converging pointwise to $\mathsf{G}_{s,0}(x)=\frac{\Lambda_s}{\;|x|^{(3-s)}}$ (see \eqref{eq:sfGlambda_asympt}-\eqref{eq:sfGlambda_asympt-2} and \eqref{eq:KalphanegEV} above); the $L^2_{\mathrm{loc}}\!\setminus\! L^2$-function $\mathsf{G}_{s,0}$ can be actually regarded as a \emph{zero-energy resonance} for $\mathsf{k}^{(s/2)}_{\alpha=0}$ (the local square-integrability following from $s\in(\frac{3}{2},\frac{5}{2})$).

\section{Approximation scheme in dimension one}\label{sec:1D}

In this Section we consider the singular perturbations of the one-dimensional fractional Laplacian and their approximation by means of fractional Schr\"odinger operators with shrinking potentials.

As for the 3D case, for $\lambda>0$ and $s\geqslant 0$, we set
\begin{equation}\label{eq:sfG1D}
  \mathsf{G}_{s,\lambda}(x)\;:=\;\frac{1}{\:(2\pi)^{\frac{1}{2}}}\Big(\frac{1}{|p|^{s}+\lambda}\Big)^{\!\vee}(x)\,,\qquad x,p\in\mathbb{R}
\end{equation}
(whence $((-\Delta)^{s/2}+\lambda) \,\mathsf{G}_{s,\lambda}=\delta(x)$ distributionally), and
\begin{equation}
 \mathring{\mathsf{k}}^{(s/2)}\;:=\;\overline{(-\Delta)^{s/2}\upharpoonright C^\infty_0(\mathbb{R}\setminus \{0\})}
\end{equation}
as an operator closure with respect to the Hilbert space $L^2(\mathbb{R})$. $\mathring{\mathsf{k}}^{(s/2)}$ has deficiency index $n\in\mathbb{N}$ when $s\in I_n^{(1)}:=(n-\frac{1}{2},n+\frac{1}{2})$ (see \eqref{eq:rule} above), and in the case of deficiency index 1 one has
\begin{equation}\label{eq:domkrings}
 \mathcal{D}(\mathring{\mathsf{k}}^{(s/2)})\;=\;\Big\{f\in H^s(\mathbb{R})\,\Big|\!\int_{\mathbb{R}}\widehat{f}(p)\,\ud p=0\Big\}\,, \qquad s\in\textstyle(\frac{1}{2},\frac{3}{2})
\end{equation}
(which can be seen by means of a completely analogous argument to that of \cite[Appendix A]{MOS-2018-fractional-and-perturbation}). The corresponding one-(real-)parameter family of self-adjoint extensions is given by the one-dimensional analogous of Theorem \ref{thm:Kalpha} \cite{MOS-2018-fractional-and-perturbation}.

\begin{theorem}\label{thm:Kalpha1D} Let $s\in(\frac{1}{2},\frac{3}{2})$ and
\begin{equation}\label{eq:Theta1D}
\Theta(s,\lambda)\;:=\;
\begin{cases}
{\textstyle \big( \lambda^{1-\frac1s}s\sin{(\frac{\pi}{s})}\big)^{-1}}&s\neq 1\\
\quad -\frac{1}{\pi}\ln\lambda&s=1\,,
\end{cases}\qquad\qquad \lambda>0\,.
\end{equation}
\begin{itemize}
 \item[(i)] The self-adjoint extensions in $L^2(\mathbb{R})$ of the operator $\mathring{\mathsf{k}}^{(s/2)}$ form the family $(\mathsf{k}^{(s/2)}_\alpha)_{\alpha\in\mathbb{R}\cup\{\infty\}}$, where $\mathsf{k}^{(s/2)}_\infty$ is its Friedrichs extension, namely the self-adjoint fractional Laplacian $(-\Delta)^{s/2}$, and all other (proper) extensions are given, for arbitrary $\lambda>0$, by
  \begin{equation}\label{eq:domKwithalpha1D}
 \begin{split}
  \mathcal{D}(\mathsf{k}^{(s/2)}_\alpha)\;&=\;\left\{\!\!
  \begin{array}{l}
   g=F^\lambda+{\displaystyle\frac{F^\lambda(0)}{\,\alpha-\Theta(s,\lambda)}}\,\mathsf{G}_{s,\lambda} \\
   \qquad F^\lambda\in H^s(\mathbb{R})
  \end{array}\!\!\right\} \\
  (\mathsf{k}^{(s/2)}_\alpha+\lambda)\,g\;&=\;((-\Delta)^{s/2}+\lambda)\,F^\lambda\,.
 \end{split}
\end{equation}
  \item[(ii)] For  each $\alpha\in\mathbb{R}$ the quadratic form of the extension $\mathsf{k}^{(s/2)}_\alpha$ is given by
\begin{eqnarray}
   \mathcal{D}[\mathsf{k}^{(s/2)}_\alpha]\!&=&\!\! H^{\frac{s}{2}}(\mathbb{R})\dotplus \mathrm{span}\{\mathsf{G}_{s,\lambda}\}  \label{eq:Kring-alpha_form1_1D}\\
 \qquad \mathsf{k}^{(s/2)}_\alpha[F^\lambda+\kappa_\lambda \mathsf{G}_{s,\lambda}]\!\!&=&\!\! \||\nabla|^s F^\lambda\|_{L^2(\mathbb{R})}^2-\lambda\|F^\lambda+\kappa_\lambda \mathsf{G}_{s,\lambda}\|_{L^2(\mathbb{R})}^2 \nonumber \\
 & & \!\!\!\!\!\!\!\!\!\!\!\!\!\!\!\!\!\!\!\!\!\!\!\!+\,\lambda\|F^\lambda\|_{L^2(\mathbb{R})}^2  +\big(\alpha-\Theta(s,\lambda)\big)|\kappa_\lambda|^2 \label{eq:Kring-alpha_form2_1D}
 \end{eqnarray}
 for arbitrary $\lambda>0$.
  \item[(iii)] The resolvent of $\mathsf{k}^{(s/2)}_\alpha$ is given by
  \begin{equation}\label{eq:Kalpha_res_1D}
   \begin{split}
      (\mathsf{k}^{(s/2)}_\alpha+\lambda\mathbbm{1})^{-1}\;&=\;((-\Delta)^{s/2}+\lambda\mathbbm{1})^{-1} \\
      &\quad +\big(\alpha-\Theta(s,\lambda)\big)^{\!-1}\, |\mathsf{G}_{s,\lambda}\rangle\langle \mathsf{G}_{s,\lambda}|
   \end{split}
  \end{equation}
  for arbitrary $\lambda>0$.
 \item[(iv)] For each $\alpha\in\mathbb{R}$ the extension $\mathsf{k}^{(s/2)}_\alpha$ is semi-bounded from below, and
 \begin{equation}\label{eq:spec-kalpha_1D}
\sigma_{\mathrm{ess}}(\mathsf{k}^{(s/2)}_\alpha)\;=\;\sigma_{\mathrm{ac}}(\mathsf{k}^{(s/2)}_\alpha)\;=  \;[0,+\infty)\,,\qquad \sigma_{\mathrm{sc}}(\mathsf{k}^{(s/2)}_\alpha)\;=\;\emptyset\,, 
\end{equation}
\begin{equation}\label{eq:spec-kalpha_mar1D}
\sigma_\mathrm{disc}(\mathsf{k}^{(s/2)}_\alpha)\;=\;
   \begin{cases}
    \quad \emptyset & \textrm{ if } s\neq 1,\,(s-1)\,\alpha\leqslant 0 \\
    \{-E_\alpha^{(s)}\} & \textrm{ if } s\neq 1,\, (s-1)\,\alpha> 0 \\
\{-E_\alpha^{(1)}\} & \textrm{ if } s=1\,,

%     \\
%      \{-E_\alpha^{(0)}\} & \textrm{ if } s=1, \alpha\neq +\infty,
   \end{cases}
\end{equation}
 where the eigenvalue $-E_\alpha^{(s)}$ is non-degenerate and is given by
 \begin{equation}\label{eq:KalphanegEV_1D}
E_\alpha^{(s)}\;=
\begin{cases}
\big(\alpha s\sin({\textstyle\frac{\pi}{s})}\big)^{\frac{s}{1-s}}&s\neq 1\,\\
\qquad e^{-\pi\alpha}&s=1\, ,
\end{cases}
\end{equation}
 the (non-normalised) eigenfunction being $\mathsf{G}_{s,\lambda=|E_\alpha^{(s)}|}$.
\end{itemize}
\end{theorem}

Our goal is to qualify each of the extensions given by Theorem \ref{thm:Kalpha1D} as suitable limits of approximating fractional Schr\"{o}dinger operators with finite range potentials. Unlike the 3D setting, here the regime $s\in(\frac{1}{2},\frac{3}{2})$ is separated by the transition value $s=1$, below which we are in the \emph{locally singular case} for the Green function \eqref{eq:sfG1D}, and above which we are in the \emph{locally regular case}, in the terminology of Section \ref{intro}. This will result in different assumptions on the approximating potentials and different schemes for the resolvent limit.

We therefore proceed by splitting our discussion into the two above-mentioned cases.

\subsection{Locally singular, resonance-driven case}~

This is the regime $s\in(\frac{1}{2},1)$. The Green function $\mathsf{G}_{s,\lambda}$ has a local singularity,
\begin{equation}\label{eq:sfGlambda_asympt1D}
 \mathsf{G}_{s,\lambda}(x)\;=\;\frac{\Lambda_s}{\;|x|^{(1-s)}}+\mathsf{J}_{s,\lambda}(x) 
\end{equation}
with
\begin{equation}\label{eq:sfGlambda_asympt-2-1D}
 \begin{split}
  \Lambda_s\;&:=\;\frac{\Gamma(\frac{1-s}{2})}{\,(2\pi)^{\frac{1}{2}}\,2^{s-\frac{1}{2}}\,\Gamma(\frac{s}{2})} \\
  \mathsf{J}_{s,\lambda}\;&:=\;-\frac{\lambda}{\:(2\pi)^{\frac{1}{2}}}\Big(\frac{1}{|p|^s(|p|^{s}+\lambda)}\Big)^{\!\vee}\;\in\;C_\infty(\mathbb{R})\,.
 \end{split}
\end{equation}
We make the following assumption (the class ${R}_{s,d}$ was introduced in \eqref{eq:Rclass}).

\medskip

\textbf{Assumption (I$^-_s$).} $s\in(\frac{1}{2},1)$ and moreover:
\begin{itemize}
 \item[(i)] $V:\mathbb{R}\to\mathbb{R}$ is a measurable function in $L^1(\mathbb{R},\langle x\rangle^{2s-1}\ud x)\cap\mathcal{R}_{s,1}$;
 \item[(ii)] $\eta:\mathbb{R}\to\mathbb{R}^+$ is a continuous function satisfying $\eta(0)=\eta(1)=1$ and
 \[
  \eta(\varepsilon)\;=\;1+\eta_s\,\varepsilon^{1-s}+o(\varepsilon^{1-s})\qquad\textrm{as }\varepsilon\downarrow 0
 \]
 for some $\eta_s\in\mathbb{R}$ that we call the \emph{strength} of the \emph{distortion factor} $\eta$.
\end{itemize}

\medskip

For given $V$ and $\eta$ satisfying Assumption ($\mathrm{I}^-_s$), let us set 
\begin{equation}\label{eq:heps-Veps1D}
 h_\varepsilon^{(s/2)}\;:=\;(-\Delta)^{s/2}+V_\varepsilon\,,\qquad V_\varepsilon(x)\;:=\;\frac{\,\eta(\varepsilon)}{\:\varepsilon^s}\,V(\textstyle\frac{x}{\varepsilon})\,,\qquad\varepsilon>0\,.
\end{equation}
For every $\varepsilon>0$ the operator $h_\varepsilon^{(s/2)}$, defined as a form sum, is self-adjoint on $L^2(\mathbb{R}^3)$ and $\sigma_{\mathrm{ess}}(h_\varepsilon^{(s/2)})=[0,+\infty)$ (Lemma \ref{lem:compactness1D}(iii)).

The zero-energy spectral behaviour of $(-\Delta)^{s/2}+V$, which is crucial for the limit $\varepsilon\downarrow 0$ in $h_\varepsilon^{(s/2)}$, is described as follows, in analogy with Theorem \ref{thm:resonance}.

\begin{theorem}\label{thm:resonance1D}
 Let $s\in(\frac{1}{2},1)$, $V\in L^1(\mathbb{R},\langle x\rangle^{2s-1}\ud x)\cap\mathcal{R}_{s,1}$, real-valued. Let $v:=|V|^{\frac{1}{2}}$ and $u:=|V|^{\frac{1}{2}}\mathrm{sign}(V)$.
 \begin{itemize}
  \item[(i)] The operator $u(-\Delta)^{-\frac{s}{2}}v$ is compact on $L^2(\mathbb{R})$.
 \end{itemize}
 Assume in addition that 
 \begin{equation}\label{eq:phi1D}
  u(-\Delta)^{-\frac{s}{2}}v\,\phi\;=\;-\phi\qquad\textrm{for some } \phi\in L^2(\mathbb{R})\setminus \{0\}
 \end{equation}
 and define
 \begin{equation}\label{eq:psi1D}
  \psi\;:=\;(-\Delta)^{-\frac{s}{2}}v\,\phi\,.
 \end{equation}
 Then:
 \begin{itemize}
  \item[(ii)] $\psi\in L^2_{\mathrm{loc}}(\mathbb{R})$ and  $\big((-\Delta)^{s/2}+V\big)\psi=0$ in the sense of distributions,
  \item[(iii)] $\langle v,\phi\rangle_{L^2}\;=\;-\displaystyle\int_{\mathbb{R}}\ud x\,V(x)\psi(x)$,
  \item[(iv)] $\psi\in L^2(\mathbb{R})$ $\Leftrightarrow$ $\langle v,\phi\rangle_{L^2}=0$, in which case $\psi\in\mathcal{D}((-\Delta)^{s/2}+V)$.
 \end{itemize}
\end{theorem}

We defer to Section \ref{sec:zero_en_res} the proof of Theorem \ref{thm:resonance1D} and a discussion of the occurrence of a zero-energy resonance for $(-\Delta)^{s/2}+V$ (Proposition \ref{prop:exist_res_1D}). With the same terminology of Section \ref{sec:3D}, $(-\Delta)^{s/2}+V$ is \emph{zero-energy resonant} and that $\psi$ is a \emph{zero-energy resonance} for $(-\Delta)^{s/2}+V$ when there exists a non-zero $L^2$-function $\phi$ satisfying \eqref{eq:phi1D} and the corresponding function $\psi$ defined by \eqref{eq:psi1D} belongs to $L^2_{\mathrm{loc}}(\mathbb{R})\setminus L^2(\mathbb{R}^)$. If, for the zero-energy resonant operator $(-\Delta)^{s/2}+V$, the eigenvalue $-1$ of $u(-\Delta)^{-\frac{s}{2}}v$ is non-degenerate, then the resonance is \emph{simple}.

Here below is our first main result in dimension one, relative to the resonance-driven regime.

\begin{theorem}\label{thm:shrinking1D}
 Let $s\in(\frac{1}{2},1)$. Given a potential $V$ and a distortion factor $\eta$ with strength $\eta_s$ satisfying Assumption \emph{($\mathrm{I}^-_s$)}, for every $\varepsilon>0$ let $h_\varepsilon^{(s/2)}=(-\Delta)^{s/2}+V_\varepsilon$ be the corresponding self-adjoint Schr\"odinger operator defined in \eqref{eq:heps-Veps1D} with the $\frac{s}{2}$-th fractional Laplacian and the shrinking potential $V_\varepsilon$.
\begin{itemize}
 \item[(i)] If $(-\Delta)^{s/2}+V$ is not zero-energy resonant, then $h_\varepsilon^{(s/2)}\xrightarrow[]{\;\varepsilon\downarrow 0\;}(-\Delta)^{s/2}$ in the norm-resolvent sense on $L^2(\mathbb{R})$.
 \item[(ii)] If $(-\Delta)^{s/2}+V$ admits a \emph{simple} zero-energy resonance $\psi$, then for  
 \[
  \alpha\;:=\;-\eta_s\, \Big|\!\int_{\mathbb{R}}\ud x\,V(x)\psi(x)\,\Big|^{-2}
 \]
 one has $h_\varepsilon^{(s/2)}\xrightarrow[]{\;\varepsilon\downarrow 0\;}\mathsf{k}^{(s/2)}_\alpha$ in the norm-resolvent sense on $L^2(\mathbb{R})$.
\end{itemize}
\end{theorem}

We shall prove Theorem \ref{thm:shrinking1D} in Section \ref{sec:proof_approx_1D}.

The alternative in Theorem \ref{thm:shrinking1D} is completely analogous to that of Theorem \ref{thm:shrinking3D}, due to the the locally singular, resonant-driven nature of both limits: only for zero-energy resonant operators $(-\Delta)^{s/2}+V$ is the limit non-trivial.

The signature of the resonance is particularly transparent in the absence of distortion factor: when $\eta(\varepsilon)\equiv 1$ by scaling one sees that $(-\Delta)^{s/2}+V_\varepsilon$ remains zero-energy resonant for any $\varepsilon>0$, and we may regard the limit operator $\mathsf{k}^{(s/2)}_{\alpha=0}$ too as zero-energy resonant, for the negative eigenvalue of $\mathsf{k}^{(s/2)}_{\alpha}$ when $|\alpha|\neq 0$ vanishes as $|\alpha|\to 0$ and the corresponding eigenfunctions becomes (proportional to) the $L^2_{\mathrm{loc}}\!\setminus\! L^2$-function $|x|^{-(1-s)}$ (see %\eqref{eq:sfGlambda_asympt}-\eqref{eq:sfGlambda_asympt-2} and 
\eqref{eq:KalphanegEV_1D} above).

\subsection{Locally regular, resonance-independent case}~

This is the regime $s\in(1,\frac{3}{2})$. In contrast with the resonance-driven regime, \emph{no spectral requirement} is now needed on the unscaled fractional operator $(-\Delta)^{s/2}+V$ and the scaling in $V_\varepsilon$ is \emph{independent} of $s$. 
Thus, we make the following assumption.

\medskip

\textbf{Assumption (I$^+_s$).}
\begin{itemize}
 \item[(i)] $V:\mathbb{R}\to\mathbb{R}$ is a measurable function in $L^1(\mathbb{R})$.
 \item[(ii)] $\eta:\mathbb{R}^+\to\mathbb{R}^+$ is a smooth function satisfying $\eta(1)=1$.
\end{itemize}

\medskip

Correspondingly, we set 
\begin{equation}\label{eq:heps-Veps1Dplus}
 h_\varepsilon^{(s/2)}\;:=\;(-\Delta)^{s/2}+V_\varepsilon\,,\qquad V_\varepsilon(x)\;:=\;\frac{\,\eta(\varepsilon)}{\:\varepsilon}\,V(\textstyle\frac{x}{\varepsilon})\,,\qquad\varepsilon>0\,.
\end{equation}
For every $\varepsilon>0$ the operator $h_\varepsilon^{(s/2)}$, defined as a form sum, is self-adjoint on $L^2(\mathbb{R}^3)$ and $\sigma_{\mathrm{ess}}(h_\varepsilon^{(s/2)})=[0,+\infty)$ (Lemma \ref{lem:compactness1Dplus}(iii)).

Here below is our second main result in dimension one, which, as opposite to Theorem \ref{thm:shrinking1D}, takes the following form.

\begin{theorem}\label{thm:shrinking1Dplus}
 Let $s\in(1,\frac{3}{2})$. For every $\varepsilon>0$ let $h_\varepsilon^{(s/2)}=(-\Delta)^{s/2}+V_\varepsilon$ be defined according to
 Assumption (I$^+_s$) and \eqref{eq:heps-Veps1Dplus}. Then $h_\varepsilon^{(s/2)}\xrightarrow[]{\;\varepsilon\downarrow 0\;}\mathsf{k}^{(s/2)}_\alpha$ in the norm-resolvent sense on $L^2(\mathbb{R})$, where 
\[
  \alpha\;:=\;-\Big(\eta(0)\!\int_{\mathbb{R}}\ud x\,V(x)\Big)^{\!-1}.
 \]
\end{theorem}

We shall prove Theorem \ref{thm:shrinking1Dplus} in Section \ref{sec:proof_approx_1D-nores}.

\section{Convergence of the 3D limit}\label{sec:proof_approx_3D}

The goal of this Section is to prove Theorem \ref{thm:shrinking3D}.

Let us start with qualifying the following useful operator-theoretic properties.

\begin{lemma}\label{lem:compactness}
 Let $V:\mathbb{R}^3\to\mathbb{R}$ belong to $L^1(\mathbb{R}^3)\cap\mathcal{R}_{s,3}$ for some $s\in(\frac{3}{2},3)$. Then:
 \begin{itemize}
  \item[(i)] for every $\lambda\geqslant 0$, $|V|^{\frac{1}{2}}((-\Delta)^{\frac{s}{2}}+\lambda\mathbbm{1})^{-1}|V|^{\frac{1}{2}}$ is a Hilbert-Schmidt operator on $L^2(\mathbb{R}^3)$;
  \item[(ii)] $|V|^{\frac{1}{2}}\ll(-\Delta)^{\frac{s}{4}}$ in the sense of infinitesimally bounded operators;
  \item[(iii)] the operator $(-\Delta)^{\frac{s}{2}}+V$ defined as a form sum is self-adjoint and $\sigma_{\mathrm{ess}}((-\Delta)^{\frac{s}{2}}+V)=[0,+\infty)$.
 \end{itemize}
\end{lemma}

\begin{proof}
 (i) $|V|^{\frac{1}{2}}((-\Delta)^{\frac{s}{2}}+\lambda\mathbbm{1})^{-1}|V|^{\frac{1}{2}}$ acts as an integral operator with kernel
 \[
  \mathcal{K}_{s,\lambda}(x,y)\;:=\;|V(x)|^{\frac{1}{2}}\mathsf{G}_{s,\lambda}(x-y)|V(y)|^{\frac{1}{2}}\,,
 \]
 and its Hilbert-Schmidt norm is estimated as
 \[
  \begin{split}
   \Big\|&|V|^{\frac{1}{2}}((-\Delta)^{\frac{s}{2}}+\lambda\mathbbm{1})^{-1}|V|^{\frac{1}{2}} \Big\|_{\mathrm{H.S.}}^2\;=\;\iint_{\mathbb{R}^3\times\mathbb{R}^3}\!\!\ud x\,\ud y\,|\mathcal{K}_{s,\lambda}(x,y)|^2 \\
   &\leqslant\;2\Lambda_s^2\iint_{\mathbb{R}^3\times\mathbb{R}^3}\!\!\ud x\,\ud y\,\frac{|V(x)|\,|V(y)|}{\;|x|^{2(3-s)}}+2\|\mathsf{J}_{s,\lambda}\|_{L^\infty}^2\iint_{\mathbb{R}^3\times\mathbb{R}^3}\!\!\ud x\,\ud y\,|V(x)|\,|V(y)| \\
   &\leqslant\;2\Lambda_s^2\|V\|_{\mathcal{R}_{s,3}}^2+2\|\mathsf{J}_{s,\lambda}\|_{L^\infty}^2\|V\|_{L^1}^2\;<\;+\infty\,,
  \end{split}
 \]
 having used \eqref{eq:sfGlambda_asympt}-\eqref{eq:sfGlambda_asympt-2} in the second step.

 (ii) The map $\lambda\mapsto |V|^{\frac{1}{2}}((-\Delta)^{\frac{s}{2}}+\lambda\mathbbm{1})^{-1}|V|^{\frac{1}{2}}$ is continuous from $(0,+\infty)$ to the space of Hilbert-Schmidt operators, and by dominated convergence
 \[
  \lim_{\lambda\to +\infty}\iint_{\mathbb{R}^3\times\mathbb{R}^3}\!\!\ud x\,\ud y\,|V(x)|\,|\mathsf{G}_{s,\lambda}(x-y)|^2\,|V(y)|\;=\;0\,.
 \]
 Therefore, for arbitrary $\varepsilon>0$ it is possible to find $\lambda_\varepsilon>0$ large enough such that
 \[
  \begin{split}
   \varepsilon\;&\geqslant\; \Big\||V|^{\frac{1}{2}}\big((-\Delta)^{\frac{s}{2}}+\lambda_\varepsilon\mathbbm{1}\big)^{-1}|V|^{\frac{1}{2}} \Big\|_{\mathrm{H.S.}}^2 \\
   &=\; \Big\|\big((-\Delta)^{\frac{s}{2}}+\lambda_\varepsilon\mathbbm{1}\big)^{-\frac{1}{2}}\,|V|\, \big((-\Delta)^{\frac{s}{2}}+\lambda_\varepsilon\mathbbm{1}\big)^{-\frac{1}{2}}\Big\|_{\mathrm{H.S.}}^2 \\
   &\geqslant\;\Big\|\big((-\Delta)^{\frac{s}{2}}+\lambda_\varepsilon\mathbbm{1}\big)^{-\frac{1}{2}}\,|V|\, \big((-\Delta)^{\frac{s}{2}}+\lambda_\varepsilon\mathbbm{1}\big)^{-\frac{1}{2}}\Big\|_{\mathrm{op}}^2\,,
  \end{split}
 \]
which implies, for some $b_\varepsilon>0$,
\[
 |\langle \varphi,V\varphi\rangle_{L^2}|\;\leqslant\;\varepsilon\;\langle\varphi,(-\Delta)^{\frac{s}{2}}\varphi\rangle_{L^2}+b_\varepsilon \|\varphi\|_{L^2}^2\qquad\forall\varphi\in\mathcal{D}[(-\Delta)^{\frac{s}{2}}]=H^{\frac{s}{2}}(\mathbb{R}^3)\,,
\]
and hence $|V|^{\frac{1}{2}}\ll(-\Delta)^{\frac{s}{4}}$.

(iii) The statement follows at once from (ii). 
\end{proof}

For chosen $s\in(\frac{3}{2},\frac{5}{2})$, $\varepsilon>0$, and $V$ and $\eta$ satisfying Assumption (I$_s$), let us recall from \eqref{eq:heps-Veps} that $V_\varepsilon(x)=\frac{\,\eta(\varepsilon)}{\:\varepsilon^s}\,V(\textstyle\frac{x}{\varepsilon})$ and let us define
\begin{equation}
 \begin{split}
   v(x)\;&:=\;|V(x)|^{\frac{1}{2}}\,,\qquad u(x)\;:=\;|V(x)|^{\frac{1}{2}}\,\mathrm{sign}(V(x))\,, \\
   v_\varepsilon(x)\;&:=\;|V_\varepsilon(x)|^{\frac{1}{2}}\,,\;\quad u_\varepsilon(x)\;:=\;|V_\varepsilon(x)|^{\frac{1}{2}}\,\mathrm{sign}(V_\varepsilon(x))\,.
   \end{split}
\end{equation}
Thus,
\begin{equation}\label{eq:veps_ueps}
 v_\varepsilon(x)\;=\;\textstyle\frac{\sqrt{\eta(\varepsilon)}}{\:\varepsilon^{s/2}}\,v({\frac{x}{\varepsilon}})\,,\qquad  u_\varepsilon(x)\;=\;\textstyle\frac{\sqrt{\eta(\varepsilon)}}{\:\varepsilon^{s/2}}\,u({\frac{x}{\varepsilon}})\,,\qquad v_\varepsilon u_\varepsilon\;=\;V_\varepsilon\,.
\end{equation}

The Hamiltonian $ h_\varepsilon^{(s/2)}=(-\Delta)^{s/2}+V_\varepsilon$ defined in \eqref{eq:heps-Veps} as a form sum is self-adjoint on $L^2(\mathbb{R}^3)$, as guaranteed by Lemma \ref{lem:compactness}(iii). An expression for its resolvent that is convenient in the present context is the Konno-Kuroda identity \cite{konno-kuroda-1966}. One has the following.

\begin{lemma}\label{lem:konnokuroda}
Let $V:\mathbb{R}^3\to\mathbb{R}$ belong to $L^1(\mathbb{R}^3)\cap\mathcal{R}_{s,3}$ for some $s\in(\frac{3}{2},\frac{5}{2})$. Then
 \begin{equation}\label{eq:KK}
  \begin{split}
   &\big( h_\varepsilon^{(s/2)}+\lambda\mathbbm{1}\big)^{-1}\;=\;\big((-\Delta)^{s/2}+\lambda\mathbbm{1}\big)^{-1} \;- \\
   &-\big((-\Delta)^{s/2}+\lambda\mathbbm{1}\big)^{-1} v_\varepsilon \Big(\mathbbm{1}+u_\varepsilon\big((-\Delta)^{s/2}+\lambda\mathbbm{1}\big)^{-1} v_\varepsilon \Big)^{\!-1}   u_\varepsilon\,\big((-\Delta)^{s/2}+\lambda\mathbbm{1}\big)^{-1}\!\!\!\!\!\!\!\!
  \end{split}
 \end{equation}
 for every $\varepsilon>0$ and every $-\lambda<0$ in the resolvent set of $ h_\varepsilon^{(s/2)}$, as an identity between bounded operators on $L^2(\mathbb{R}^3)$.
\end{lemma}

\begin{proof}
 The statement is precisely the application of the Konno-Kuroda resolvent identity, for which we follow the formulation presented in \cite[Theorem B.1(b)]{albeverio-solvable}, to the operator $(-\Delta)^{s/2}+v_\varepsilon u_\varepsilon$. For the validity of such identity two conditions are needed: 
 the compactness of $u_\varepsilon((-\Delta)^{s/2}+\lambda\mathbbm{1})^{-1} v_\varepsilon$ and
 the infinitesimal bound $|V|^{\frac{1}{2}}\ll(-\Delta)^{\frac{s}{4}}$. Both conditions are guaranteed by Lemma \ref{lem:compactness}.
\end{proof}

Observe that the invertibility of $\mathbbm{1}+u_\varepsilon((-\Delta)^{s/2}+\lambda\mathbbm{1})^{-1}v_\varepsilon$ (with bounded inverse) is part of the statement of the Konno-Kuroda formula \eqref{eq:KK}.

It is convenient to manipulate the identity \eqref{eq:KK} further so as to isolate terms in the r.h.s.~which are easily controllable in the limit $\varepsilon\downarrow 0$. To this aim, let us introduce for each $\varepsilon>0$ the unitary scaling operator $U_\varepsilon: L^2(\mathbb{R}^3)\to L^2(\mathbb{R}^3)$ defined by
\begin{equation}\label{eq:Ueps}
 (U_\varepsilon f)(x)\;:=\;\frac{1}{\:\varepsilon^{3/2}}\,f(\textstyle\frac{x}{\varepsilon})\,.
\end{equation}
Its adjoint clearly acts as $(U_\varepsilon^*f)(x)=\varepsilon^{3/2}f(\varepsilon x)$. $U_\varepsilon$ induces the scaling transformations
\begin{equation}\label{eq:scaling_transforms}
\begin{split}
 U_\varepsilon^* v_\varepsilon U_\varepsilon\;&=\;\frac{\sqrt{\eta(\varepsilon)}}{\:\varepsilon^{s/2}}\,v\,,\qquad U_\varepsilon^* u_\varepsilon U_\varepsilon\;=\;\frac{\sqrt{\eta(\varepsilon)}}{\:\varepsilon^{s/2}}\,u\,, \\
 U_\varepsilon^* \big((-\Delta)^{s/2}+\lambda\mathbbm{1} \big)^{-1} U_\varepsilon\;&=\;\varepsilon^{s}\big((-\Delta)^{s/2}+\lambda\varepsilon^s\mathbbm{1} \big)^{-1}\,,
\end{split}
\end{equation}
whose proof is straightforward.

Let us also introduce, for each $\varepsilon>0$ and for each $\mu>0$ such that $-\mu^s$ belongs to the resolvent set of $ h_\varepsilon^{(s/2)}$, the operators
\begin{equation}\label{eq:ABCops}
 \begin{split}
  A_\varepsilon^{(s)}\;&:=\;\varepsilon^{-\frac{3-s}{2}}\big((-\Delta)^{s/2}+\mu^s\mathbbm{1} \big)^{-1}(\eta(\varepsilon))^{-\frac{1}{2}}\, v_\varepsilon\, U_\varepsilon \\
  B_\varepsilon^{(s)}\;&:=\;\eta(\varepsilon)\,u\,\big((-\Delta)^{s/2}+(\mu\varepsilon)^s\mathbbm{1}\big)^{-1} v\\
  C_\varepsilon^{(s)}\;&:=\;U_\varepsilon^*\,u_\varepsilon\,(\eta(\varepsilon))^{-\frac{1}{2}}\big((-\Delta)^{s/2}+\mu^s\mathbbm{1} \big)^{-1}\varepsilon^{-\frac{3-s}{2}} \,.
 \end{split}
\end{equation}

We shall see in a moment (Lemma \ref{lem:ABClimits}) that $A_\varepsilon^{(s)}$, $B_\varepsilon^{(s)}$, and $C_\varepsilon^{(s)}$ are Hilbert-Schmidt operators on $L^2(\mathbb{R}^3)$. Most importantly for our purposes, the resolvent of $h_\varepsilon^{(s/2)}$ takes the following convenient form.

\begin{lemma}\label{ABCresolvent}
 Under the present assumptions,
 \begin{equation}\label{eq:ABCresolvent}
  \big( h_\varepsilon^{(s/2)}+\mu^s\mathbbm{1}\big)^{-1}\;=\;\big((-\Delta)^{s/2}+\mu^s\mathbbm{1}\big)^{-1}-A_\varepsilon^{(s)} \varepsilon^{3-s}\eta(\varepsilon)(\mathbbm{1}+B_\varepsilon^{(s)})^{-1} C_\varepsilon^{(s)}
\end{equation}
for every $\varepsilon>0$ and every $\mu>0$ such that $-\mu^s$ belongs to the resolvent set of $ h_\varepsilon^{(s/2)}$.
\end{lemma}

\begin{proof}
 In formula \eqref{eq:KK} we set $\lambda=\mu^s$ and we insert $\mathbbm{1}=U_\varepsilon U_\varepsilon^*$ in the second summand of the r.h.s.~right after $((-\Delta)^{s/2}+\lambda\mathbbm{1})^{-1} v_\varepsilon$. We then commute $U_\varepsilon^*$ all the way through by means of the scaling transformations \eqref{eq:scaling_transforms}: this way, we reproduce the product $A_\varepsilon^{(s)} \varepsilon^{3-s}\eta(\varepsilon)(\mathbbm{1}+B_\varepsilon^{(s)})^{-1} C_\varepsilon^{(s)}$. 
\end{proof}

The limit $\varepsilon\downarrow 0$ can be monitored explicitly for $A_\varepsilon^{(s)}$, $B_\varepsilon^{(s)}$, and $C_\varepsilon^{(s)}$.

\begin{lemma}\label{lem:ABClimits}
 For every $\varepsilon>0$, $A_\varepsilon^{(s)}$, $B_\varepsilon^{(s)}$, and $C_\varepsilon^{(s)}$ are Hilbert-Schmidt operators on $L^2(\mathbb{R}^3)$ with limit
 \begin{eqnarray}
     \lim_{\varepsilon\downarrow 0}A_\varepsilon^{(s)}\!\!&=&\!\!|\mathsf{G}_{s,\mu^s}\rangle\langle v|  \label{eq:Aeps_limit}\\
        \lim_{\varepsilon\downarrow 0}B_\varepsilon^{(s)}\!\!&=&\!\!B_0^{(s)}\;=\;u\,(-\Delta)^{-\frac{s}{2}} v \label{eq:Beps_limit} \\
   \lim_{\varepsilon\downarrow 0}C_\varepsilon^{(s)}\!\!&=&\!\!|u\rangle\langle \mathsf{G}_{s,\mu^s}| \label{eq:Ceps_limit}
 \end{eqnarray}
 in the Hilbert-Schmidt operator norm. 
\end{lemma}

\begin{proof}
 By construction, see \eqref{eq:veps_ueps}, \eqref{eq:Ueps}, and \eqref{eq:ABCops} above, 
 \begin{equation*}
  \begin{split}
   (A_\varepsilon^{(s)}f)(x)\;&=\;\varepsilon^{-\frac{3-s}{2}}\,\varepsilon^{-\frac{s}{2}}\,\varepsilon^{-\frac{3}{2}}\int_{\mathbb{R}^3}\mathsf{G}_{s,\mu^s}(x-y)\,v({\textstyle\frac{y}{\varepsilon}})f({\textstyle\frac{y}{\varepsilon}})\,\ud y \\
   &=\;\int_{\mathbb{R}^3}\mathsf{G}_{s,\mu^s}(x-\varepsilon y)\,v(y)f(y)\,\ud y \qquad\quad \forall f\in L^2(\mathbb{R}^3)\,,
  \end{split}
 \end{equation*}
 that is,  $A_\varepsilon^{(s)}$ acts as an integral operator with kernel $\mathsf{G}_{s,\mu^s}(x-\varepsilon y)v(y)$\,. The latter is clearly a function in $L^2(\mathbb{R}^3\times\mathbb{R}^3,\ud x\,\ud y)$ uniformly in $\varepsilon$, and dominated convergence implies
 \[
  \|A_\varepsilon^{(s)}\|_{\mathrm{H.S.}}^2\;=\;\iint_{\mathbb{R}^3\times\mathbb{R}^3}\!\!\ud x\,\ud y\,|\mathsf{G}_{s,\mu^s}(x-\varepsilon y)v(y)|^2\;\xrightarrow[]{\;\varepsilon\downarrow 0\;}\;\|\mathsf{G}_{s,\mu^s}\|_{L^2}^2\|V\|_{L^1}
 \]
 as well as
 \[
  \begin{split}
   \langle g, A_\varepsilon^{(s)}f\rangle_{L^2}\;&=\;\iint_{\mathbb{R}^3\times\mathbb{R}^3}\!\!\ud x\,\ud y\,\overline{g(x)}\,\mathsf{G}_{s,\mu^s}(x-\varepsilon y)\,v(y)f(y) \\
   &\xrightarrow[]{\;\varepsilon\downarrow 0\;}\;\langle g,\mathsf{G}_{s,\mu^s}\rangle_{L^2}\langle v,f\rangle_{L^2}\qquad \forall f,g\in C^\infty_0(\mathbb{R}^3)\,.
  \end{split}
 \]
 As a consequence, as $\varepsilon\downarrow 0$, $A_\varepsilon^{(s)}\to |\mathsf{G}_{s,\mu^s}\rangle\langle v|$ \emph{weakly} in the operator topology, \emph{and} the Hilbert-Schmidt norm of $A_\varepsilon^{(s)}$ converges to the Hilbert-Schmidt norm of its limit. By a well-known feature of compact operators  \cite[Theorem 2.21]{simon_trace_ideals}, the combination of these two properties implies that $A_\varepsilon^{(s)}\to |\mathsf{G}_{s,\mu^s}\rangle\langle v|$ in the Hilbert-Schmidt topology. This proves \eqref{eq:Aeps_limit}.

 The discussion for $C_\varepsilon^{(s)}$ is completely analogous: its integral kernel is $u(x)\mathsf{G}_{s,\mu^s}(\varepsilon x-y)$ and \eqref{eq:Ceps_limit} is proved by the very same type of argument.

 Concerning $B_\varepsilon^{(s)}$, its integral kernel is $\eta(\varepsilon)u(x) \mathsf{G}_{s,(\mu\varepsilon)^s}(x-y)v(y)$ and the integral kernel of $B_0^{(s)}$  is $u(x) \mathsf{G}_{s,0}(x-y)v(y)$: owing to Lemma \ref{lem:compactness}(i) both operators are Hilbert Schmidt, and moreover by dominated convergence $B_\varepsilon^{(s)}\to B_0^{(s)}$  weakly in the operator topology and $\|B_\varepsilon^{(s)} \|_{\mathrm{H.S.}}^2\to\|B_0^{(s)}\|_{\mathrm{H.S.}}^2$ as $\varepsilon\downarrow 0$. By the same property \cite[Theorem 2.21]{simon_trace_ideals} the limit \eqref{eq:Beps_limit} then holds in the Hilbert-Schmidt norm. 
\end{proof}

It is evident from \eqref{eq:ABCresolvent} that, in order for the limits \eqref{eq:Aeps_limit}--\eqref{eq:Ceps_limit} above to qualify the behaviour of the resolvent of $h_\varepsilon^{(s/2)}$ as $\varepsilon\downarrow 0$, one needs additional information on the possible failure of invertibility in $L^2(\mathbb{R}^3)$ of the operator $\mathbbm{1}+B_0^{(s)}$. By the Fredholm alternative, since $B_0^{(s)}$ is compact, $(\mathbbm{1}+B_0^{(s)})^{-1}$ exists everywhere defined and bounded, in which case \eqref{eq:ABCresolvent} implies at once $\big( h_\varepsilon^{(s/2)}+\mu^s\mathbbm{1}\big)^{-1}\to\big((-\Delta)^{s/2}+\mu^s\mathbbm{1}\big)^{-1}$  as $\varepsilon\downarrow 0$, \emph{unless} $B_0^{(s)}$ admits an eigenvalue $-1$.

Let us then assume that the latter circumstance does occurs, namely condition \eqref{eq:phi} of Theorem \ref{thm:resonance}. More precisely, we make the following assumption.

\medskip

\textbf{Assumption (II$_s$).} Assumption (I$_s$) holds. $B_0^{(s)}$ has eigenvalue $-1$, which is non-degenerate. $\phi\in L^2(\mathbb{R}^3)$ is a non-zero function such that $B_0^{(s)}\phi=-\phi$ and, in addition, $\langle \widetilde{\phi},\phi\rangle_{L^2}=-1$, where $\widetilde{\phi}:=(\mathrm{sign}V)\phi$.

\medskip

Since $\langle \widetilde{\phi},\phi\rangle_{L^2}=-\langle(\mathrm{sign}V)\phi,(\mathrm{sign}V)v(-\Delta)^{-\frac{s}{2}}v\phi\rangle_{L^2}=-\|(-\Delta)^{-\frac{s}{4}}v\phi\|_{L^2}^2$, the normalisation $\langle \widetilde{\phi},\phi\rangle_{L^2}=-1$ is always possible.

Under Assumption (II$_s$), $(\mathbbm{1}+B_\varepsilon^{(s)})^{-1}$ becomes singular in the limit $\varepsilon\downarrow 0$, with a singularity that now competes with the vanishing factor $\varepsilon^{3-s}$ of \eqref{eq:ABCresolvent}. To resolve this competing effect, we need first an expansion of $B_\varepsilon^{(s)}$ around $\varepsilon=0$ to a further order, than the limit \eqref{eq:Beps_limit}. This expansion holds irrespectively of Assumption (II$_s$).

\begin{lemma}\label{lem:GB_expansion} Let $s\in(\frac{3}{2},\frac{5}{2})$ and $\lambda>0$.
\begin{itemize}
 \item[(i)] For every $x\in\mathbb{R}^3\!\setminus\! \{0\}$
 \begin{equation}\label{eq:limG}
 \lim_{\lambda\downarrow 0}\,\frac{\mathsf{G}_{s,\lambda}(x)-\mathsf{G}_{s,0}(x)}{\;(2\pi s \sin(\frac{3\pi}{s}))^{-1}\lambda^{\frac{3}{s}-1}}\;=\;1\,.
\end{equation}
 \item[(ii)] In the norm operator topology one has
 \begin{equation}\label{eq:Bexpansion}
  \lim_{\varepsilon\downarrow 0}\,\frac{1}{\;(\mu\varepsilon)^{3-s}}\,\big(B_\varepsilon^{(s)}-B_0^{(s)}\big)\;=\;\frac{\:\eta_s}{\:\mu^{3-s}}\, B_0^{(s)}+ \frac{1}{\:2\pi s \sin(\frac{3\pi}{s})}\,|u\rangle\langle v|\,.
 \end{equation}
 Here $\mu>0$ is the constant chosen in the definition \eqref{eq:ABCops} of $B_\varepsilon^{(s)}$ and $\eta_s\in\mathbb{R}$ is the constant that is part of Assumption \emph{(I$_s$)}.
\end{itemize}
\end{lemma}

\begin{proof}
 (i) From \eqref{eq:sfG} we write
 \[
 \begin{split}
  \frac{\mathsf{G}_{s,\lambda}(x)-\mathsf{G}_{s,0}(x)}{\;\lambda^{\frac{3}{s}-1}}\;&=\;\frac{1}{\:\lambda^{\frac{3}{s}-1}(2\pi)^{\frac{3}{2}}}\int_{\mathbb{R}^3}\ud p\;e^{\ii x\cdot p}\,\frac{-\lambda}{\;(2\pi)^{\frac{3}{2}}|p|^s(|p|^s+\lambda)} \\
  &=\;-\frac{1}{\:(2\pi)^{3}}\int_{\mathbb{R}^3}\ud p\;e^{\ii\,\lambda^{1/s} x\cdot p}\,\frac{1}{\;|p|^s(|p|^s+1)}\,,
 \end{split}
 \]
 whence
\[
 \frac{\mathsf{G}_{s,\lambda}(x)-\mathsf{G}_{s,0}(x)}{\;\lambda^{\frac{3}{s}-1}}\;\xrightarrow[]{\;\lambda\downarrow 0\;} \;-\frac{1}{\:(2\pi)^{3}}\int_{\mathbb{R}^3}\ud p\;\frac{1}{\;|p|^s(|p|^s+1)}\;=\;\frac{1}{\:2\pi s \sin(\frac{3\pi}{s})}
\]
by dominated convergence, since $p\mapsto(|p|^s(|p|^s+1))^{-1}$ is integrable when $s\in(\frac{3}{2},3)$. 
 
 (ii) The Hilbert-Schmidt operator
 \[
  \frac{1}{\;(\mu\varepsilon)^{3-s}}\,\big(B_\varepsilon^{(s)}-B_0^{(s)}\big)-\frac{\:\eta_s}{\:\mu^{3-s}}\, B_0^{(s)}- \frac{1}{\:2\pi s \sin(\frac{3\pi}{s})}\,|u\rangle\langle v|
 \]
 has integral kernel
 \[\tag{*}
  \begin{split}
   u(x)&\Big(\,\frac{\eta(\varepsilon)-1}{\:(\mu\varepsilon)^{3-s}}-\frac{\:\eta_s}{\:\mu^{3-s}}\Big)\,\mathsf{G}_{s,(\mu\varepsilon)^s}(x-y)\,v(y)\;+ \\
   &\qquad + u(x)\Big(\,\frac{\,\mathsf{G}_{s,(\mu\varepsilon)^s}(x-y)-\mathsf{G}_{s,0}(x-y)}{(\mu\varepsilon)^{3-s}}-\frac{1}{\:2\pi s \sin(\frac{3\pi}{s})}\Big)\,v(y)\,.
  \end{split}
 \]
 The first summand in (*) vanishes as $\varepsilon\downarrow 0$ for a.e.~$x,y\in\mathbb{R}^3$ as a consequence of Assumption (I$_s$)(ii), and so does the second summand in (*) as a consequence of \eqref{eq:limG}, where we take $\lambda=(\mu\varepsilon)^s$. Moreover, each such summand belongs to $L^2(\mathbb{R}^3\times\mathbb{R}^3,\ud x\,\ud y)$ uniformly in $\varepsilon$, thanks to the assumption (I$_s$)(i) on the potentials $v$ and $u$. Thus, by dominated convergence, the function (*) vanishes in $L^2(\mathbb{R}^3\times\mathbb{R}^3,\ud x\,\ud y)$ as $\varepsilon\downarrow 0$, and this proves the limit \eqref{eq:Bexpansion} in the Hilbert-Schmidt norm. 
\end{proof}

We can now monitor the competing effect in $\varepsilon^{3-s}(\mathbbm{1}+B_\varepsilon^{(s)})^{-1}$ as $\varepsilon\downarrow 0$. 

\begin{lemma}\label{lem:competing_effect}
 Under the Assumptions \emph{(I$_s$)} and \emph{(II$_s$)} one has
 \begin{equation}\label{eq:central_part_limit}
  \lim_{\varepsilon\downarrow 0}\;(\mu\varepsilon)^{3-s}\big(\mathbbm{1}+B_\varepsilon^{(s)}\big)^{-1}\;=\;\Big(\frac{\eta_s}{\:\mu^{3-s}}+\frac{\,|\langle v,\phi\rangle_{L^2}|^2}{\,2\pi s\sin\frac{3\pi}{s}} \Big)^{\!-1}\:|\phi\rangle\langle\widetilde{\phi}|
 \end{equation}
 in the operator norm topology. 
\end{lemma}

\begin{proof}
 We re-write \eqref{eq:Bexpansion} in the form of the expansion
 \[\tag{i}
  B_\varepsilon^{(s)}\;=\;B_0^{(s)}+(\mu\varepsilon)^{3-s}\mathcal{B}^{(s)}+o(\varepsilon^{3-s})
 \]
where, for short,
 \[
  \mathcal{B}^{(s)}\;:=\;\frac{\:\eta_s}{\:\mu^{3-s}}\, B_0^{(s)}+ \frac{1}{\:2\pi s \sin(\frac{3\pi}{s})}\,|u\rangle\langle v|\,,
 \]
whence also 
 \[\tag{ii}
  \begin{split}
   (\mu\varepsilon)^{3-s}&\big(\mathbbm{1}+B_\varepsilon^{(s)}\big)^{-1}\;= \\
   &=\;\Big(\mathbbm{1}+(\mu\varepsilon)^{3-s}\big(\mathbbm{1}+(\mu\varepsilon)^{3-s}+B_0^{(s)} \big)^{-1}\big(\mathcal{B}^{(s)}-\mathbbm{1}+o(1)\big)\Big)^{\!-1}\;\times \\
   &\qquad\qquad\times\; (\mu\varepsilon)^{3-s}\big(\mathbbm{1}+(\mu\varepsilon)^{3-s}+B_0^{(s)} \big)^{-1}\,.
  \end{split}
 \]
The $o(\varepsilon^a)$-remainders in (i) and (ii) above are clearly meant in the Hilbert-Schmidt norm.

The operator $(\mu\varepsilon)^{3-s}(\mathbbm{1}+(\mu\varepsilon)^{3-s}+B_0^{(s)})^{-1}$ that appears twice in (ii) is of the form
\[
 z(\mathbbm{1}+T+z\mathbbm{1})^{-1}\,,\qquad z\in\mathbb{C}\setminus \{0\}\,,
\]
for a closed operator $T$ with \emph{isolated} eigenvalue $-1$; this is a general setting for which a well-known expansion by Kato is available as $z\to 0$ \cite[Sec. 3.6.5]{Kato-perturbation}, which in the present context (in complete analogy with the argument of the proof of \cite[Lemma I.1.2.4]{albeverio-solvable}) reads
\[\tag{iii}
 (\mu\varepsilon)^{3-s}\big(\mathbbm{1}+(\mu\varepsilon)^{3-s}+B_0^{(s)} \big)^{-1}\;=\;-|\phi\rangle\langle\widetilde{\phi}|+O(\varepsilon^{3-s})
\]
as $\varepsilon\downarrow 0$ in the operator norm topology. In practice, $(\mathbbm{1}+(\mu\varepsilon)^{3-s}+B_0^{(s)})^{-1}$ remains bounded also in the limit $\varepsilon\downarrow 0$ when restricted to the orthogonal complement of the eigenspace $-1$ of $B_0^{(s)}$, whereas it becomes singular when restricted to such eigenspace; the magnitude of the singularity is precisely $(\mu\varepsilon)^{-(3-s)}$, which is cancelled exactly by the pre-factor $(\mu\varepsilon)^{3-s}$ in the l.h.s.~of (iii). In fact, by assumption of non-degeneracy, the eigenspace $-1$ is spanned by $\phi$ and $P:=-|\phi\rangle\langle\widetilde{\phi}|$ projects onto $\mathrm{span}\{\phi\}$ with $P\phi=\phi$, as follows from the normalisation $\langle \widetilde{\phi},\phi\rangle_{L^2}=-1$.

Combining (ii) and (iii) above yields
\[\tag{iv}
 (\mu\varepsilon)^{3-s}\big(\mathbbm{1}+B_\varepsilon^{(s)}\big)^{-1}\;=\;\big(\mathbbm{1}+P(\mathcal{B}^{(s)}-\mathbbm{1})+O(\varepsilon^{3-s})\big)^{-1}(P+O(\varepsilon^{3-s}))
\]
as $\varepsilon\downarrow 0$ in the operator norm topology.

Next, in order to see that the limit $\varepsilon\downarrow 0$ in the r.h.s.~of (iv) exists and is a bounded operator, we write explicitly 
\[\tag{v}
 \begin{split}
  \mathbbm{1}+P(\mathcal{B}^{(s)}-\mathbbm{1})\;&=\;\mathbbm{1}-|\phi\rangle\langle\widetilde{\phi}|\Big(\frac{\:\eta_s}{\:\mu^{3-s}}\, u(-\Delta)^{-\frac{s}{2}}v+ \frac{1}{\:2\pi s \sin(\frac{3\pi}{s})}\,|u\rangle\langle v|-\mathbbm{1}\Big) \\
  &=\;\mathbbm{1}+\frac{\:\eta_s}{\:\mu^{3-s}}|\phi\rangle\langle\widetilde{\phi}|-\frac{\overline{\langle v,\phi\rangle_{L^2}}}{\:2\pi s\sin(\frac{3\pi}{s})\:}\,|\phi\rangle\langle v|+|\phi\rangle\langle\widetilde{\phi}|\,,
 \end{split}
\]
 where we used the identities $\langle\widetilde{\phi},u\rangle_{L^2}=\langle\phi,v\rangle_{L^2}$ and 
 \[
 \begin{split}
  \big\langle\widetilde{\phi}, u(-\Delta)^{-\frac{s}{2}}v f\big\rangle_{L^2}\;&=\;\big\langle v(-\Delta)^{-\frac{s}{2}} u\,\widetilde{\phi},f\big\rangle_{L^2}\;=\;\big\langle(\mathrm{sign}V) u(-\Delta)^{-\frac{s}{2}} v\phi,f\big\rangle_{L^2} \\
  &=\;-\langle\widetilde{\phi},f\rangle_{L^2}\qquad\forall f\in L^2(\mathbb{R}^3)\,.
 \end{split}
 \]
Setting the constants
\[
 \begin{split}
  a\;&:=\;\Big(\frac{\eta_s}{\:\mu^{3-s}}+1\Big)\Big(\frac{\eta_s}{\:\mu^{3-s}}+\frac{\:|\langle v,\phi\rangle_{L^2}|^2}{\:2\pi s\sin\frac{3\pi}{s}\,}\Big)^{\!-1} \\
  b\;&:=\;-\frac{\overline{\langle v,\phi\rangle_{L^2}}}{\:2\pi s\sin\frac{3\pi}{s}\,}\Big(\frac{\eta_s}{\:\mu^{3-s}}+\frac{\:|\langle v,\phi\rangle_{L^2}|^2}{\:2\pi s\sin\frac{3\pi}{s}\,}\Big)^{\!-1},
 \end{split}
\]
the expression (v) allows one to compute explicitly (using again $\langle\widetilde{\phi},\phi\rangle_{L^2}=-1$)
\[
 \big( \mathbbm{1}+P(\mathcal{B}^{(s)}-\mathbbm{1}) \big)\;\big(\mathbbm{1}+a\,|\phi\rangle\langle\widetilde{\phi}|+b\,|\phi\rangle\langle v|\big)\;=\;\mathbbm{1}
\]
and therefore to deduce that $(\mathbbm{1}+P(\mathcal{B}^{(s)}-\mathbbm{1}))^{-1}$ exists and is bounded. This fact allows one to deduce from (iv) that
\[\tag{vi}
 \lim_{\varepsilon\downarrow 0}\;(\mu\varepsilon)^{3-s}\big(\mathbbm{1}+B_\varepsilon^{(s)}\big)^{-1}\;=\;\big(\mathbbm{1}+P(\mathcal{B}^{(s)}-\mathbbm{1}))\big)^{-1}P
\]
in the operator norm topology.

Last, from (v), using $\langle\widetilde{\phi},\phi\rangle_{L^2}=-1$ and $\langle\widetilde{\phi},u\rangle_{L^2}=\langle\phi,v\rangle_{L^2}$, one finds
\[
  \big(\mathbbm{1}+P(\mathcal{B}^{(s)}-\mathbbm{1})\big)\phi\;=\;-\Big(\frac{\eta_s}{\:\mu^{3-s}}+\frac{\,|\langle v,\phi\rangle_{L^2}|^2}{\,2\pi s\sin\frac{3\pi}{s}} \Big)\,\phi
\]
and hence
\[
  \big(\mathbbm{1}+P(\mathcal{B}^{(s)}-\mathbbm{1})\big)^{-1}\phi\;=\;-\Big(\frac{\eta_s}{\:\mu^{3-s}}+\frac{\,|\langle v,\phi\rangle_{L^2}|^2}{\,2\pi s\sin\frac{3\pi}{s}} \Big)^{\!-1}\,\phi\,.
\]
Plugging the latter identity into (vi) yields finally \eqref{eq:central_part_limit} as a limit in the operator norm.
\end{proof}

We are now in the condition to prove Theorem \ref{thm:shrinking3D}.

\begin{proof}[Proof of Theorem \ref{thm:shrinking3D}]
 Owing to \eqref{eq:ABCresolvent} we need to determine the limit of 
 \[
  -A_\varepsilon^{(s)} \varepsilon^{3-s}\eta(\varepsilon)(\mathbbm{1}+B_\varepsilon^{(s)})^{-1} C_\varepsilon^{(s)}
 \]
 as $\varepsilon\downarrow 0$.
%  
%  
%  
%   \begin{equation}
%   \big( h_\varepsilon^{(s/2)}+\mu^s\mathbbm{1}\big)^{-1}\;=\;\big((-\Delta)^{s/2}+\mu^s\mathbbm{1}\big)^{-1}-A_\varepsilon^{(s)} \varepsilon^{3-s}\eta(\varepsilon)(\mathbbm{1}+B_\varepsilon^{(s)})^{-1} C_\varepsilon^{(s)}
% \end{equation}
%  
As observed already, if $u(-\Delta)^{-\frac{s}{2}}v$ has no eigenvalue $-1$, then the above expression vanishes with $\varepsilon$ and
\[
 \big( h_\varepsilon^{(s/2)}+\mu^s\mathbbm{1}\big)^{-1}\;\xrightarrow[]{\;\varepsilon\downarrow 0\;}\;\big((-\Delta)^{s/2}+\mu^s\mathbbm{1}\big)^{-1}
\]
in the operator norm. If instead $u(-\Delta)^{-\frac{s}{2}}v$ does admit a simple eigenvalue $-1$, be $(-\Delta)^{\frac{s}{2}}+V$ zero-energy resonant or not, we are under the Assumption (I$_s$) and (II$_s$) of the present Section and we can therefore apply the limits \eqref{eq:Aeps_limit}, \eqref{eq:Ceps_limit}, and \eqref{eq:central_part_limit}. This yields
\[\tag{*}
 \begin{split}
   -A_\varepsilon^{(s)} \varepsilon^{3-s}&\eta(\varepsilon)(\mathbbm{1}+B_\varepsilon^{(s)})^{-1} C_\varepsilon^{(s)} \\
   \xrightarrow[]{\;\varepsilon\downarrow 0\;}&\;-|\mathsf{G}_{s,\mu^s}\rangle\langle v| \circ
   \Big(\eta_s+\frac{\:\mu^{3-s}|\langle v,\phi\rangle_{L^2}|^2}{\,2\pi s\sin\frac{3\pi}{s}} \Big)^{\!-1}|\phi\rangle\langle\widetilde{\phi}|   \circ
   |u\rangle\langle \mathsf{G}_{s,\mu^s}| \\
   =&\;-\displaystyle\frac{|\langle v,\phi\rangle_{L^2}|^2}{\displaystyle\:\eta_s+\frac{\:\mu^{3-s}|\langle v,\phi\rangle_{L^2}|^2}{\,2\pi s\sin\frac{3\pi}{s}} \:}\;|\mathsf{G}_{s,\mu^s}\rangle\langle \mathsf{G}_{s,\mu^s}|
 \end{split}
\]
in the operator norm,
having used $\langle\widetilde{\phi},u\rangle_{L^2}=\langle\phi,v\rangle_{L^2}$. Now, if $(-\Delta)^{\frac{s}{2}}+V$ is \emph{not} zero-energy resonant, then $\langle v,\phi\rangle_{L^2}=0$, owing to Theorem \ref{thm:resonance}(iv), and the conclusion is again
 \[
 \big( h_\varepsilon^{(s/2)}+\mu^s\mathbbm{1}\big)^{-1}\;\xrightarrow[]{\;\varepsilon\downarrow 0\;}\;\big((-\Delta)^{s/2}+\mu^s\mathbbm{1}\big)^{-1}
\]
in the operator norm. This proves part (i) of the present Theorem. If instead $(-\Delta)^{\frac{s}{2}}+V$ is zero-energy resonant, then using $\langle v,\phi\rangle_{L^2}\neq 0$ and plugging (*) back into \eqref{eq:ABCresolvent} yields
\[
 \begin{split}
  \big( h_\varepsilon^{(s/2)}+\mu^s\mathbbm{1}\big)^{-1}\;\xrightarrow[]{\;\varepsilon\downarrow 0\;}\;\big(&(-\Delta)^{s/2}+\mu^s\mathbbm{1}\big)^{-1} \\
  &\qquad+\displaystyle\frac{1}{\displaystyle\:\frac{-\eta_s}{\:|\langle v,\phi\rangle_{L^2}|^2\,}-\frac{\:\mu^{3-s}}{\,2\pi s\sin\frac{3\pi}{s}} \:}\;|\mathsf{G}_{s,\mu^s}\rangle\langle \mathsf{G}_{s,\mu^s}|
 \end{split}
\]
in the operator norm. Upon setting $\alpha:=-\eta_s|\langle v,\phi\rangle_{L^2}|^{-2}$ and $\lambda=\mu^s$, and comparing the resulting expression with \eqref{eq:Kalpha_res}, this means
\[
 \begin{split}
  \big( h_\varepsilon^{(s/2)}+\lambda\mathbbm{1}\big)^{-1}\;\xrightarrow[]{\;\varepsilon\downarrow 0\;}&\;\big((-\Delta)^{s/2}+\lambda\mathbbm{1}\big)^{-1}+{\textstyle\big(\alpha-\frac{\lambda^{\frac{3}{s}-1}}{2\pi s\sin(\frac{3\pi}{s})}\big)^{\!-1}}\, |\mathsf{G}_{s,\lambda}\rangle\langle \mathsf{G}_{s,\lambda}| \\
  &=\;(\mathsf{k}^{(s/2)}_\alpha+\lambda\mathbbm{1})^{-1}\,,
 \end{split}
\]
which proves part (ii) of the Theorem.
\end{proof}

% \begin{acknowledgements} We are deeply indebted to G.~Dell'Antonio for many enlightening discussions on the subject and to V.~Georgiev for his precious advices on his work \cite{Fujiwara-Georgiev-Ozawa-2016} and on Theorem \ref{KatoPonce_Vladimir}.
% \end{acknowledgements}

\section{Convergence of the 1D limit: resonant-driven case.}\label{sec:proof_approx_1D}

The proof of the limit $h_\varepsilon^{(s/2)}\xrightarrow[]{\;\varepsilon\downarrow 0\;}\mathsf{k}^{(s/2)}_\alpha$ in dimension one when $s\in(\frac{1}{2},1)$ (Theorem \ref{thm:shrinking1D}) is technically analogous to that in three dimensions. Therefore, based on the detailed discussion of the preceding Section, we only present here the steps of the convergence scheme and a sketch of their proofs.

Prior to that, let us set up the key resolvent identity and useful scaling properties with a notation that we can use also in Section \ref{sec:proof_approx_1D-nores} when we will deal with the resonant-independent limit.

We then keep $s\in(\frac{1}{2},1)\cup(1,\frac{3}{2})$ generic for a moment and, in a unified form, we re-write \eqref{eq:heps-Veps1D} and \eqref{eq:heps-Veps1Dplus} as
\begin{equation}\label{eq:heps-Veps1D_unified}
 V_\varepsilon(x)\;=\;\frac{\,\eta(\varepsilon)}{\:\varepsilon^{\frac{s+\gamma}{2}}}\,V(\textstyle\frac{x}{\varepsilon})\,.
\end{equation}
Taking $\gamma=s$ in \eqref{eq:heps-Veps1D_unified} yields \eqref{eq:heps-Veps1D} and taking $\gamma=2-s$ yields \eqref{eq:heps-Veps1Dplus}. Thus, setting
\begin{equation}\label{eq:1dvu}
 \begin{split}
   v(x)\;&:=\;|V(x)|^{\frac{1}{2}}\,,\qquad u(x)\;:=\;|V(x)|^{\frac{1}{2}}\,\mathrm{sign}(V(x))\,, \\
   v_\varepsilon(x)\;&:=\;|V_\varepsilon(x)|^{\frac{1}{2}}\,,\;\quad u_\varepsilon(x)\;:=\;|V_\varepsilon(x)|^{\frac{1}{2}}\,\mathrm{sign}(V_\varepsilon(x))\,,
   \end{split}
\end{equation}
one has
\begin{equation}\label{eq:veps_ueps1D}
 v_\varepsilon(x)\;=\;\textstyle\frac{\sqrt{\eta(\varepsilon)}}{\:\varepsilon^{(s+\gamma)/4}}\,v({\frac{x}{\varepsilon}})\,,\qquad  u_\varepsilon(x)\;=\;\textstyle\frac{\sqrt{\eta(\varepsilon)}}{\:\varepsilon^{(s+\gamma)/2}}\,u({\frac{x}{\varepsilon}})\,,\qquad v_\varepsilon u_\varepsilon\;=\;V_\varepsilon\,.
\end{equation}

The 1D analogue $U_\varepsilon: L^2(\mathbb{R})\to L^2(\mathbb{R})$ of the unitary scaling operator \eqref{eq:Ueps} acts as
\begin{equation}\label{eq:Ueps1D}
 (U_\varepsilon f)(x)\;:=\;{\textstyle\frac{1}{\:\varepsilon^{1/2}}}\,f(\textstyle\frac{x}{\varepsilon})\,,
\end{equation}
which induces the scaling transformations
\begin{equation}\label{eq:scaling_transforms1D}
\begin{split}
 U_\varepsilon^* v_\varepsilon U_\varepsilon\;&=\;\frac{\sqrt{\eta(\varepsilon)}}{\:\varepsilon^{\frac{s+\gamma}{4}}}\,v\,,\qquad U_\varepsilon^* u_\varepsilon U_\varepsilon\;=\;\frac{\sqrt{\eta(\varepsilon)}}{\:\varepsilon^{\frac{s+\gamma}{4}}}\,u\,, \\
 U_\varepsilon^* \big((-\Delta)^{s/2}+\lambda\mathbbm{1} \big)^{-1} U_\varepsilon\;&=\;\varepsilon^{s}\big((-\Delta)^{s/2}+\lambda\varepsilon^s\mathbbm{1} \big)^{-1}\,.
\end{split}
\end{equation}

Based on arguments that differ depending on whether $s\in(\frac{1}{2},1)$ or $s\in(1,\frac{3}{2})$ and which we shall prove in due time, the Konno-Kuroda-type resolvent identity 
 \begin{equation}\label{eq:KK_1D}
  \begin{split}
   &\big( h_\varepsilon^{(s/2)}+\lambda\mathbbm{1}\big)^{-1}\;=\;\big((-\Delta)^{s/2}+\lambda\mathbbm{1}\big)^{-1} \;- \\
   &-\big((-\Delta)^{s/2}+\lambda\mathbbm{1}\big)^{-1} v_\varepsilon \Big(\mathbbm{1}+u_\varepsilon\big((-\Delta)^{s/2}+\lambda\mathbbm{1}\big)^{-1} v_\varepsilon \Big)^{\!-1}   u_\varepsilon\,\big((-\Delta)^{s/2}+\lambda\mathbbm{1}\big)^{-1}\!\!\!\!\!\!\!\!
  \end{split}
 \end{equation}
holds as an identity between bounded operators on $L^2(\mathbb{R})$ for every $\varepsilon>0$ and every $-\lambda<0$ in the resolvent set of $ h_\varepsilon^{(s/2)}$. Inserting $U_\varepsilon U_\varepsilon^*=\mathbbm{1}$ into \eqref{eq:KK_1D} and applying \eqref{eq:scaling_transforms1D} then yields
 \begin{equation}\label{eq:ABCresolvent_1D}
  \big( h_\varepsilon^{(s/2)}+\lambda\mathbbm{1}\big)^{-1}=\;\big((-\Delta)^{s/2}+\lambda\mathbbm{1}\big)^{-1}-A_\varepsilon^{(s)} \varepsilon^{\frac{2-s-\gamma}{2}}\eta(\varepsilon)(\mathbbm{1}+B_\varepsilon^{(s)})^{-1} C_\varepsilon^{(s)}\,,
\end{equation}
having defined
\begin{equation}\label{eq:ABCops_1D}
 \begin{split}
  A_\varepsilon^{(s)}\;&:=\;\varepsilon^{-\frac{2-s-\gamma}{2}}\big((-\Delta)^{s/2}+\lambda\mathbbm{1} \big)^{-1}(\eta(\varepsilon))^{-\frac{1}{2}}\, v_\varepsilon\, U_\varepsilon \\
  B_\varepsilon^{(s)}\;&:=\;\eta(\varepsilon)\,\varepsilon^{\frac{s-\gamma}{2}}u\,\big((-\Delta)^{s/2}+\lambda \varepsilon^s\mathbbm{1}\big)^{-1} v\\
  C_\varepsilon^{(s)}\;&:=\;U_\varepsilon^*\,u_\varepsilon\,(\eta(\varepsilon))^{-\frac{1}{2}}\big((-\Delta)^{s/2}+\lambda\mathbbm{1} \big)^{-1}\varepsilon^{-\frac{2-s-\gamma}{2}} \,.
 \end{split}
\end{equation}
We shall see in a moment (Lemma \ref{lem:ABClimits1D}) that $A_\varepsilon^{(s)}$, $B_\varepsilon^{(s)}$, and $C_\varepsilon^{(s)}$ are Hilbert-Schmidt operators on $L^2(\mathbb{R})$.

The following scaling property too is going to be useful in both regimes $s\in(\frac{1}{2},1)$ and $s\in(1,\frac{3}{2})$.

\begin{lemma}
 For any $s,\gamma,\varepsilon>0$ and any $x\in\mathbb{R}\!\setminus\!\{0\}$ one has
 \begin{equation}\label{eq:1DscalingProperty}
  \varepsilon^{\frac{s-\gamma}{2}}\mathsf{G}_{s,\lambda\varepsilon^s}(x)\;=\;\varepsilon^{\frac{2-s-\gamma}{2}}\mathsf{G}_{s,\lambda}(\varepsilon x)\,.
 \end{equation}
\end{lemma}

\begin{proof}
 Owing to \eqref{eq:sfG1D},
 \[
  \begin{split}
   \varepsilon^{\frac{s-\gamma}{2}}\mathsf{G}_{s,\lambda\varepsilon^s}(x)\;&=\;\frac{1}{2\pi}\,\varepsilon^{\frac{s-\gamma}{2}}\int_{\mathbb{R}}\ud p\,e^{\ii px}\,\frac{1}{\,|p|^s+\lambda\varepsilon^s} \\
   &=\;\frac{1}{2\pi}\,\varepsilon^{\frac{2-s-\gamma}{2}}\int_{\mathbb{R}}\ud p\,e^{\ii p(\varepsilon x)}\,\frac{1}{\,|p|^s+\lambda}\;=\;\varepsilon^{\frac{2-s-\gamma}{2}}\mathsf{G}_{s,\lambda}(\varepsilon x)\,,
  \end{split}
 \]
whence the thesis.
\end{proof}

We can now start the discussion for the proof of Theorem \ref{thm:shrinking1D}, thus working in the regime $s\in(\frac{1}{2},1)$.

First, we have the following properties.
\begin{lemma}\label{lem:compactness1D}
 Let $V:\mathbb{R}\to\mathbb{R}$ belong to $L^1(\mathbb{R})\cap\mathcal{R}_{s,1}$ for some $s\in(\frac{1}{2},1)$. Then:
 \begin{itemize}
  \item[(i)] for every $\lambda\geqslant 0$, $|V|^{\frac{1}{2}}((-\Delta)^{\frac{s}{2}}+\lambda\mathbbm{1})^{-1}|V|^{\frac{1}{2}}$ is a Hilbert-Schmidt operator on $L^2(\mathbb{R})$;
  \item[(ii)] $|V|^{\frac{1}{2}}\ll(-\Delta)^{\frac{s}{4}}$ in the sense of infinitesimally bounded operators;
  \item[(iii)] the operator $(-\Delta)^{\frac{s}{2}}+V$ defined as a form sum is self-adjoint and $\sigma_{\mathrm{ess}}((-\Delta)^{\frac{s}{2}}+V)=[0,+\infty)$.
 \end{itemize}
\end{lemma}

\begin{proof}
 The proof is completely analogous to that of Lemma \ref{lem:compactness} for the 3D case, and is based on the fact that the integral kernel of $|V|^{\frac{1}{2}}((-\Delta)^{\frac{s}{2}}+\lambda\mathbbm{1})^{-1}|V|^{\frac{1}{2}}$, namely $|V(x)|^{\frac{1}{2}}\mathsf{G}_{s,\lambda}(x-y)|V(y)|^{\frac{1}{2}}$, belongs to $L^2(\mathbb{R}\times\mathbb{R},\ud x\,\ud y)$, as a direct consequence of the assumption $V\in L^1(\mathbb{R})\cap\mathcal{R}_{s,1}$. 
\end{proof}

Lemma \ref{lem:compactness1D} justifies the validity of the resolvent identity \eqref{eq:KK_1D}, and hence of the rescaled identity \eqref{eq:ABCresolvent_1D}, owing again to the general argument of \cite[Theorem B.1(b)]{albeverio-solvable}.

Next, we monitor separately the following limits.

\begin{lemma}\label{lem:ABClimits1D}
 Let $V$ and $\eta$ satisfy Assumption $(\mathrm{I}^-_s)$ for some $s\in(\frac{1}{2},1)$. For every $\varepsilon>0$, the operators $A_\varepsilon^{(s)}$, $B_\varepsilon^{(s)}$, and $C_\varepsilon^{(s)}$ defined by \eqref{eq:heps-Veps1D_unified}-\eqref{eq:veps_ueps1D} and \eqref{eq:ABCops_1D} with $\gamma=s$ are Hilbert-Schmidt operators on $L^2(\mathbb{R})$ with limit
 \begin{eqnarray}
     \lim_{\varepsilon\downarrow 0}A_\varepsilon^{(s)}\!\!&=&\!\!|\mathsf{G}_{s,\lambda}\rangle\langle v|  \label{eq:Aeps_limit1D}\\
        \lim_{\varepsilon\downarrow 0}B_\varepsilon^{(s)}\!\!&=&\!\!B_0^{(s)}\;=\;u\,(-\Delta)^{-\frac{s}{2}} v \label{eq:Beps_limit1D} \\
   \lim_{\varepsilon\downarrow 0}C_\varepsilon^{(s)}\!\!&=&\!\!|u\rangle\langle \mathsf{G}_{s,\lambda}| \label{eq:Ceps_limit1D}
 \end{eqnarray}
 in the Hilbert-Schmidt operator norm. 
\end{lemma}

\begin{proof}
 Completely analogous to the proof of Lemma \ref{lem:ABClimits}, the integral kernels being now (with $\gamma=s$)
 \[
  \begin{split}
   A_\varepsilon^{(s)}(x,y)\;&=\;\mathsf{G}_{s,\lambda}(x-\varepsilon y) \,v(y) \\
   B_\varepsilon^{(s)}(x,y)\;&=\;\eta(\varepsilon)\,u(x)\,\mathsf{G}_{s,\lambda\varepsilon^s}(x-y)\,v(y)  \\
   C_\varepsilon^{(s)}(x,y)\;&=\;u(x)\,\mathsf{G}_{s,\lambda}(\varepsilon x-y)\,.
  \end{split}
 \]
In particular, owing to \eqref{eq:1DscalingProperty},
\[
 B_\varepsilon^{(s)}(x,y)\;=\;\eta(\varepsilon)\,\varepsilon^{1-s}u(x)\,\mathsf{G}_{s,\lambda}(\varepsilon x-\varepsilon y)\,v(y)\,,
\]
and using \eqref{eq:sfGlambda_asympt1D}-\eqref{eq:sfGlambda_asympt-2-1D}
% re-writing \eqref{eq:sfG1D} as
% \[
%  \mathsf{G}_{s,\lambda}(x)\;=\;{\textstyle\frac{\,2^{1-s}\Gamma(\frac{1-s}{2})}{(2\pi)^{\frac{1}{2}}\Gamma(\frac{s}{2})}}\,\frac{1}{\;|x|^{1-s}}-\frac{\lambda}{\,(2\pi)^{\frac{1}{2}}}\Big(\frac{1}{\,|p|^s(|p|^s+\lambda}\Big)^{\vee}(x)
% \]
one finds
\[
  B_\varepsilon^{(s)}(x,y)\;\xrightarrow[]{\;\varepsilon\downarrow 0\;}\;u(x)\,{\textstyle\frac{\,2^{1-s}\Gamma(\frac{1-s}{2})}{(2\pi)^{\frac{1}{2}}\Gamma(\frac{s}{2})}}\,\frac{1}{\;|x-y|^{1-s}}\,v(y)\;=\;B_0^{(s)}(x,y)
\]
pointwise almost everywhere. 
\end{proof}

Before plugging the limits found in Lemma \eqref{lem:ABClimits1D} into \eqref{eq:ABCresolvent_1D}, that now reads
 \begin{equation}\label{eq:ABCresolvent_1D_s12-1}
  \big( h_\varepsilon^{(s/2)}+\lambda\mathbbm{1}\big)^{-1}=\;\big((-\Delta)^{s/2}+\lambda\mathbbm{1}\big)^{-1}-A_\varepsilon^{(s)} \varepsilon^{1-s}\eta(\varepsilon)(\mathbbm{1}+B_\varepsilon^{(s)})^{-1} C_\varepsilon^{(s)}\,,
\end{equation}
we see that, since $B_0^{(s)}$ is compact, $(\mathbbm{1}+B_0^{(s)})^{-1}$ exists everywhere defined and bounded, in which case $\big( h_\varepsilon^{(s/2)}+\lambda\mathbbm{1}\big)^{-1}\to\big((-\Delta)^{s/2}+\lambda\mathbbm{1}\big)^{-1}$  as $\varepsilon\downarrow 0$, \emph{unless} $B_0^{(s)}$ admits an eigenvalue $-1$. We then consider the following additional assumption.

\medskip

\textbf{Assumption (II$^-_s$).} Assumption ($\mathrm{I}^-_s$) holds. $B_0^{(s)}$ has eigenvalue $-1$, which is non-degenerate. $\phi\in L^2(\mathbb{R}^3)$ is a non-zero function such that $B_0^{(s)}\phi=-\phi$ and, in addition, $\langle \widetilde{\phi},\phi\rangle_{L^2}=-1$, where $\widetilde{\phi}:=(\mathrm{sign}V)\phi$.

\medskip

Since $\langle \widetilde{\phi},\phi\rangle_{L^2}=-\langle(\mathrm{sign}V)\phi,(\mathrm{sign}V)v(-\Delta)^{-\frac{s}{2}}v\phi\rangle_{L^2}=-\|(-\Delta)^{-\frac{s}{4}}v\phi\|_{L^2}^2$, the normalisation $\langle \widetilde{\phi},\phi\rangle_{L^2}=-1$ is always possible.

When Assumption ($\mathrm{II}^-_s$) holds, $(\mathbbm{1}+B_\varepsilon^{(s)})^{-1}$ becomes singular in the limit $\varepsilon\downarrow 0$, with a singularity that now competes with the vanishing factor $\varepsilon^{1-s}$ of \eqref{eq:ABCresolvent_1D_s12-1}. To resolve this competing effect, we need first to expand $B_\varepsilon^{(s)}$ around $\varepsilon=0$ to a further order, than the limit \eqref{eq:Beps_limit1D}. This expansion is valid irrespectively of Assumption ($\mathrm{II}^-_s$).

\begin{lemma} Let $s\in(\frac{1}{2},1)$ and $\lambda>0$.
\begin{itemize}
 \item[(i)] For every $x\in\mathbb{R}\!\setminus\!\{0\}$
 \begin{equation}\label{eq:limG1D}
 \lim_{\lambda\downarrow 0}\,\frac{\mathsf{G}_{s,\lambda}(x)-\mathsf{G}_{s,0}(x)}{\;(s \sin(\frac{\pi}{s}))^{-1}\lambda^{\frac{1}{s}-1}}\;=\;1\,.
\end{equation}
 \item[(ii)] In the norm operator topology one has
 \begin{equation}\label{eq:Bexpansion1D}
  \lim_{\varepsilon\downarrow 0}\,\frac{1}{\;\lambda^{\frac{1}{s}-1}\varepsilon^{1-s}\,}\,\big(B_\varepsilon^{(s)}-B_0^{(s)}\big)\;=\;\frac{\:\eta_s}{\:\lambda^{\frac{1}{s}-1}}\, B_0^{(s)}+ \frac{1}{\,s \sin(\frac{3\pi}{s})}\,|u\rangle\langle v|\,.
 \end{equation}
 Here $\eta_s\in\mathbb{R}$ is the constant that is part of Assumption $(\mathrm{I}^-_s)$.
\end{itemize}
\end{lemma}

\begin{proof}
 Completely analogous to the proof of Lemma \ref{lem:GB_expansion} for the 3D case.
\end{proof}

We can now monitor the competing effect in $\varepsilon^{1-s}(\mathbbm{1}+B_\varepsilon^{(s)})^{-1}$ as $\varepsilon\downarrow 0$. 

\begin{lemma}\label{lem:competing_effect1D}
 Under the Assumptions $(\mathrm{I}^-_s)$ and $(\mathrm{II}^-_s)$ one has
 \begin{equation}\label{eq:central_part_limit1D}
  \lim_{\varepsilon\downarrow 0}\;\varepsilon^{1-s}\big(\mathbbm{1}+B_\varepsilon^{(s)}\big)^{-1}\;=\;\Big(\eta_s+\frac{\,|\langle v,\phi\rangle_{L^2}|^2}{\,\lambda^{\frac{1}{s}-1} s\sin\frac{\pi}{s}} \Big)^{\!-1}\:|\phi\rangle\langle\widetilde{\phi}|
 \end{equation}
 in the operator norm topology. 
\end{lemma}

\begin{proof}
 Completely analogous to the proof of Lemma \ref{lem:competing_effect} for the 3D case.
\end{proof}

With these preliminaries at hand, we can prove Theorem \ref{thm:shrinking1D}.

\begin{proof}[Proof of Theorem \ref{thm:shrinking1D}]
The argument is the very same as the in the proof of Theorem \ref{thm:shrinking3D} for the 3D case. Thus, the limit is the trivial one unless the potential in the approximating operators satisfy Assumptions $(\mathrm{I}^-_s)$ and $(\mathrm{II}^-_s)$, in which case, plugging the limits \eqref{eq:Aeps_limit1D}, \eqref{eq:Ceps_limit1D}, and \eqref{eq:central_part_limit1D} into \eqref{eq:ABCresolvent_1D_s12-1}, one has
\[
 \begin{split}
  \big( h_\varepsilon^{(s/2)}+\mu^s\mathbbm{1}\big)^{-1}\;\xrightarrow[]{\;\varepsilon\downarrow 0\;}\;\big(&(-\Delta)^{s/2}+\mu^s\mathbbm{1}\big)^{-1} \\
  &\qquad+\displaystyle\frac{1}{\displaystyle\:\frac{-\eta_s}{\:|\langle v,\phi\rangle_{L^2}|^2\,}-\frac{1}{\,\lambda^{1-\frac{1}{s}} s\sin\frac{3\pi}{s}} \:}\;|\mathsf{G}_{s,\mu^s}\rangle\langle \mathsf{G}_{s,\mu^s}|\,.
 \end{split}
\]
The comparison of the limit resolvent above with formulas \eqref{eq:Theta1D} and \eqref{eq:Kalpha_res_1D} shows finally that the limit resolvent is precisely $(\mathsf{k}^{(s/2)}_\alpha+\lambda\mathbbm{1})^{-1}$ where the extension parameter satisfies $\alpha=-\eta_s|\int_{\mathbb{R}}\ud x\,V(x)\psi(x)|^{-2}$, and this completes the proof.
\end{proof}

\section{Convergence of the 1D limit: resonant-independent case.}\label{sec:proof_approx_1D-nores}

This Section contains the proof of Theorem \ref{thm:shrinking1Dplus}. Thus, now $s\in(1,\frac{3}{2})$ and formulas \eqref{eq:heps-Veps1D_unified}-\eqref{eq:1DscalingProperty} must be specialised with $\gamma=2-s$.

First, we observe that with $L^1$-potentials the following operator-theoretic properties hold.

\begin{lemma}\label{lem:compactness1Dplus}
 Let $V:\mathbb{R}\to\mathbb{R}$ belong to $L^1(\mathbb{R})$ and let $s\in(1,\frac{3}{2})$. Then:
 \begin{itemize}
  \item[(i)] for every $\lambda\geqslant 0$, $|V|^{\frac{1}{2}}((-\Delta)^{\frac{s}{2}}+\lambda\mathbbm{1})^{-1}|V|^{\frac{1}{2}}$ is a Hilbert-Schmidt operator on $L^2(\mathbb{R})$;
  \item[(ii)] $|V|^{\frac{1}{2}}\ll(-\Delta)^{\frac{s}{4}}$ in the sense of infinitesimally bounded operators;
  \item[(iii)] the operator $(-\Delta)^{\frac{s}{2}}+V$ defined as a form sum is self-adjoint and $\sigma_{\mathrm{ess}}((-\Delta)^{\frac{s}{2}}+V)=[0,+\infty)$.
 \end{itemize}
\end{lemma}

\begin{proof} Since $s>1$, \eqref{eq:sfG1D} defines a function $\widehat{\mathsf{G}_{s,\lambda}}\in L^1(\mathbb{R})$, whence $\mathsf{G}_{s,\lambda}\in C_\infty(\mathbb{R})$ (continuous and vanishing at infinity). Therefore, the integral kernel of $|V|^{\frac{1}{2}}((-\Delta)^{\frac{s}{2}}+\lambda\mathbbm{1})^{-1}|V|^{\frac{1}{2}}$, namely $|V(x)|^{\frac{1}{2}}\mathsf{G}_{s,\lambda}(x-y)|V(y)|^{\frac{1}{2}}$, belongs to $L^2(\mathbb{R}\times\mathbb{R},\ud x\,\ud y)$, and this holds for any $\lambda\geqslant 0$. Based on this observation, the rest of the reasoning of the proof of Lemma \ref{lem:compactness} can be repeated verbatim.
\end{proof}

Following again the general argument of \cite[Theorem B.1(b)]{albeverio-solvable}, Lemma \ref{lem:compactness1Dplus} justifies the validity of the resolvent identity \eqref{eq:KK_1D}, and hence of the rescaled identity \eqref{eq:ABCresolvent_1D}, that now reads
 \begin{equation}\label{eq:ABCresolvent_1D_s1-32}
  \big( h_\varepsilon^{(s/2)}+\lambda\mathbbm{1}\big)^{-1}=\;\big((-\Delta)^{s/2}+\lambda\mathbbm{1}\big)^{-1}-\eta(\varepsilon)\,A_\varepsilon^{(s)}(\mathbbm{1}+B_\varepsilon^{(s)})^{-1} C_\varepsilon^{(s)}
\end{equation}
for every $\varepsilon>0$ and every $-\lambda<0$ in the resolvent set of $ h_\varepsilon^{(s/2)}$.

\begin{lemma}\label{lem:ABClimits1Dplus}
 Let $V$ and $\eta$ satisfy Assumption $(\mathrm{I}^+_s)$ for some $s\in(1,\frac{3}{2})$. For every $\varepsilon>0$, the operators $A_\varepsilon^{(s)}$, $B_\varepsilon^{(s)}$, and $C_\varepsilon^{(s)}$ defined by \eqref{eq:heps-Veps1D_unified}-\eqref{eq:veps_ueps1D} and \eqref{eq:ABCops_1D} with $\gamma=2-s$ are Hilbert-Schmidt operators on $L^2(\mathbb{R})$ with limit
 \begin{eqnarray}
     \lim_{\varepsilon\downarrow 0}A_\varepsilon^{(s)}\!\!&=&\!\!|\mathsf{G}_{s,\lambda}\rangle\langle v|  \label{eq:Aeps_limit1Dplus}\\
        \lim_{\varepsilon\downarrow 0}B_\varepsilon^{(s)}\!\!&=&\!\!B_0^{(s)}\;=\;\frac{\eta(0)}{\,\lambda^{1-\frac{1}{s}}s\,\sin\frac{\pi}{s}}\,|u\rangle\langle v| \label{eq:Beps_limit1Dplus} \\
   \lim_{\varepsilon\downarrow 0}C_\varepsilon^{(s)}\!\!&=&\!\!|u\rangle\langle \mathsf{G}_{s,\lambda}| \label{eq:Ceps_limit1Dplus}
 \end{eqnarray}
 in the Hilbert-Schmidt operator norm. 
\end{lemma}

\begin{proof}
The integral kernels are now
 \[
  \begin{split}
   A_\varepsilon^{(s)}(x,y)\;&=\;\mathsf{G}_{s,\lambda}(x-\varepsilon y) \,v(y) \\
   B_\varepsilon^{(s)}(x,y)\;&=\;\eta(\varepsilon)\,\varepsilon^{s-1}\,u(x)\,\mathsf{G}_{s,\lambda\varepsilon^s}(x-y)\,v(y)  \\
   C_\varepsilon^{(s)}(x,y)\;&=\;u(x)\,\mathsf{G}_{s,\lambda}(\varepsilon x-y)\,.
  \end{split}
 \]
For $A_\varepsilon^{(s)}$ and $C_\varepsilon^{(s)}$ we reason precisely as in the proof of Lemma \ref{lem:ABClimits}. $B_\varepsilon^{(s)}$ is Hilbert-Schmidt as a consequence of Lemma \ref{lem:compactness1Dplus}. Re-writing
\[
 B_\varepsilon^{(s)}(x,y)\;=\;\eta(\varepsilon)\,u(x)\,\mathsf{G}_{s,\lambda}(\varepsilon x-\varepsilon y)\,v(y)
\]
by means of \eqref{eq:1DscalingProperty}, and observing that \eqref{eq:sfG1D} implies
\[
 \mathsf{G}_{s,\lambda}(\varepsilon x-\varepsilon y)\;\xrightarrow[]{\;\varepsilon\downarrow 0\;}\;\mathsf{G}_{s,\lambda}(0)\;=\;\frac{1}{\,\lambda^{1-\frac{1}{s}}s\,\sin\frac{\pi}{s}}\,,
\]
one deduces
\[
  B_\varepsilon^{(s)}(x,y)\;\xrightarrow[]{\;\varepsilon\downarrow 0\;}\;\frac{\eta(0)}{\,\lambda^{1-\frac{1}{s}}s\,\sin\frac{\pi}{s}} u(x)v(y)\,.
\]
Then a dominated convergence argument, analogous to that used in the proof of Lemma \ref{lem:ABClimits}, proves \eqref{eq:Beps_limit1Dplus}.
\end{proof}

It is now convenient to observe the following (see \cite[Lemma 5.1]{DM-2015-halfline} for an analogous argument).

\begin{lemma}
 Assume that the data $s\in(1,\frac{3}{2})$, $\lambda>0$ with $-\lambda$ in the resolvent set of all the $ h_\varepsilon^{(s/2)}$'s, and $V$ and $\eta$ matching Assumption $(\mathrm{I}^+_s)$, do \emph{not} satisfy the exceptional relation
 \begin{equation}\label{eq:except}
 1+\frac{\eta(0)}{\,\lambda^{1-\frac{1}{s}}s\,\sin\frac{\pi}{s}}\int_{\mathbb{R}}\!\ud x V(x)\;=\;0\,.
 \end{equation}
Then the operator $\mathbbm{1}+B_0^{(s)}$ is invertible with bounded inverse, everywhere defined on $L^2(\mathbb{R})$. 
\end{lemma}

\begin{proof}
 Since $B_0^{(s)}$ is compact on $L^2(\mathbb{R})$, based on the Fredholm alternative we have to prove that the validity of \eqref{eq:except} is equivalent to $B_0^{(s)}$ having eigenvalue $-1$. In fact, $B_0^{(s)}\phi=-\phi$ for some non-zero $\phi\in L^2(\mathbb{R})$ is the same as
 \[
  \phi\;=\;\frac{\eta(0)}{\,\lambda^{1-\frac{1}{s}}s\,\sin\frac{\pi}{s}}\,\langle v,\phi\rangle_{L^2}\,u\,,
 \]
meaning that $\phi$ is not orthogonal to $v$ and $\phi$ is a multiple of $u$. When this is the case, $u$ itself must be an eigenfunction of $B_0^{(s)}$ with eigenvalue $-1$, and this is tantamount, owing to the identity above, as the validity of \eqref{eq:except}. 
\end{proof}

For given $s$, $\eta$, and $V$, the exceptional value of $-\lambda$ satisfying \eqref{eq:except} is going to correspond to the negative eigenvalue of $\mathsf{k}^{(s/2)}_\alpha$ described in Theorem \ref{thm:Kalpha1D}(iv). As we are going to monitor the limit  $h_\varepsilon^{(s/2)}\xrightarrow[]{\;\varepsilon\downarrow 0\;}\mathsf{k}^{(s/2)}_\alpha$ in the \emph{resolvent sense}, not only must we discard the spectral points 
$-\lambda$ not belonging to the resolvent set of all the $ h_\varepsilon^{(s/2)}$'s, but also the point $-\lambda$ given by \eqref{eq:except}. Thus, for our purposes the operator $\mathbbm{1}+B_0^{(s)}$ is always invertible with everywhere defined bounded inverse.

In particular, \eqref{eq:Beps_limit1Dplus} implies
\begin{equation}\label{eq:1Blimit}
 (\mathbbm{1}+B_\varepsilon^{(s)})^{-1}\;\xrightarrow[]{\;\varepsilon\downarrow 0\;}\;(\mathbbm{1}+B_0^{(s)})^{-1}
\end{equation}
in the operator norm.

Based on the preceding preparatory materials, we can now prove Theorem \ref{thm:shrinking1Dplus}.

\begin{proof}[Proof of Theorem \ref{thm:shrinking1Dplus}]
 Since \eqref{eq:except} is excluded and therefore
 \[
  (\mathbbm{1}+B_0^{(s)})^{-1} u\;=\;\Big(1+\frac{\,\eta(0)\int_{\mathbb{R}}\ud x\,V(x)\,}{\,\lambda^{1-\frac{1}{s}}s\,\sin\frac{\pi}{s}}\Big)^{\!-1}u\,,
 \]
then plugging the limits \eqref{eq:Aeps_limit1Dplus}, \eqref{eq:Ceps_limit1Dplus}, and \eqref{eq:1Blimit} into \eqref{eq:ABCresolvent_1D_s1-32} yields
\[
 \begin{split}
  \big( h_\varepsilon^{(s/2)}+\lambda\mathbbm{1}\big)^{-1}\;\xrightarrow[]{\;\varepsilon\downarrow 0\;}\;\big((-\Delta)^{s/2}+\lambda\mathbbm{1}\big)^{-1}-\frac{\eta(0)\int_{\mathbb{R}}\ud x\,V(x)}{1+\frac{\,\eta(0)\int_{\mathbb{R}}\ud x\,V(x)\,}{\,\lambda^{1-\frac{1}{s}}s\,\sin\frac{\pi}{s}}}\,|\mathsf{G}_{s,\lambda}\rangle\langle \mathsf{G}_{s,\lambda}|
 \end{split}
\]
 in the operator norm. Upon setting 
 \[
  \alpha\;:=\;-\Big(\eta(0)\!\int_{\mathbb{R}}\ud x\,V(x)\Big)^{\!-1}
 \]
 and comparing the resulting expression with \eqref{eq:Theta1D} and \eqref{eq:Kalpha_res_1D}, one finds
 \[
 \begin{split}
  \big( h_\varepsilon^{(s/2)}+\lambda\mathbbm{1}\big)^{-1}\;\xrightarrow[]{\;\varepsilon\downarrow 0\;}&\;\big((-\Delta)^{s/2}+\lambda\mathbbm{1}\big)^{-1}+\frac{1}{\alpha-\frac{1}{\,\lambda^{1-\frac{1}{s}}s\,\sin\frac{\pi}{s}}}\,|\mathsf{G}_{s,\lambda}\rangle\langle \mathsf{G}_{s,\lambda}| \\
  &\;=\; (\mathsf{k}^{(s/2)}_\alpha+\lambda\mathbbm{1})^{-1}\,,
 \end{split}
 \]
which completes the proof. 
\end{proof}

\section{Zero-energy resonances for Schr\"{o}dinger operators \\ with fractional Laplacian}\label{sec:zero_en_res}

The purpose of this Section is two-fold. First, we prove Theorems \ref{thm:resonance} and \ref{thm:resonance1D}, concerning the characterisation of the zero-energy resonant behaviour of $(-\Delta)^{s/2}+V$. Then, we discuss the occurrence of zero-energy resonances, both in one and three dimension.

\begin{proof}[Proof of Theorem \ref{thm:resonance}]
 The fact that for a real-valued $V\in L^1(\mathbb{R}^3)\cap\mathcal{R}_{s,3}$ the operator $u(-\Delta)^{\frac{s}{2}}v$ is Hilbert-Schmidt follows from Lemma \ref{lem:compactness}(i), thus part (i) is proved. 
 
 Let us split
 \[\tag{a}\label{eq:split}
  \begin{split}
   \psi(x)\;&=\;((-\Delta)^{-\frac{s}{2}} v\phi)(x)\;=\;\int_{\mathbb{R}^3}\ud y\,\frac{\Lambda_s}{\:|x-y|^{3-s}}\,v(y)\phi(y) \\
   &=\;\frac{\,\Lambda_s\langle v,\phi\rangle_{L^2}}{\;|x|^{3-s}}+\Lambda_s\int_{\mathbb{R}^3}\ud y\,\Big(\frac{1}{\:|x-y|^{3-s}}-\frac{1}{\:|x|^{3-s}}\Big)v(y)\phi(y) \\
   &\equiv\;\frac{\,\Lambda_s\langle v,\phi\rangle_{L^2}}{\;|x|^{3-s}}+\psi_1(x)\,,
  \end{split}
 \]
 where $\Lambda_s$ is the constant defined in \eqref{eq:sfGlambda_asympt-2}. We now see that $\psi_1\in L^2(\mathbb{R}^3)$. To this aim, we observe that setting $\widehat{y}:=\frac{y}{|y|}$ one has
\begin{equation*}
\begin{split}
\int_{\R^3}\ud x&\,\Big(\frac{1}{\:|x-y|^{3-s}}-\frac{1}{\:|x|^{3-s}}\Big)^2\\
&=\;|y|^{2s-3}\!\int_{\R^3}\ud x\Big(\frac{|x-\widehat{y}\,|^{3-s}-|x|^{3-s}}{|x-\widehat{y}\,|^{3-s}|x|^{3-s}}\Big)^2\\
&\lesssim\;\! |y|^{2s-3}\int_{\R^3}\ud x\,\Big(\frac{\langle x\rangle^{2-s}}{|x-\widehat{y}\,|^{3-s}|x|^{3-s}}\Big)^2,\\
\end{split}
\end{equation*}
having used the change of variable $x\mapsto|y|\,x$ in the first step and the uniform bound $\big||x-\widehat{y}\,|^{3-s}-|x|^{3-s}\big|\lesssim\langle x\rangle^{2-s}$ in the last step. Since $s\in(\frac32,\frac52)$, the last integral above is finite, thus we deduce
\begin{equation*}\tag{b}\label{eq:moredecay}
\int_{\R^3}\ud x\,\Big(\frac{1}{\:|x-y|^{3-s}}-\frac{1}{\:|x|^{3-s}}\Big)^2\;\lesssim\; |y|^{2s-3}\,.
\end{equation*}
As a consequence,
\begin{equation*}
\begin{split}
\|\psi_1\|^2_{L^2(\R^3)}\;&\lesssim\; \int_{\R^3}\ud x\,\Big|\int_{\mathbb{R}^3}\ud y\,\Big(\frac{1}{\:|x-y|^{3-s}}-\frac{1}{\:|x|^{3-s}}\Big)v(y)\phi(y)\Big|^2\\
&\leqslant\;  \iint_{\R^3 \times\R^3}\ud x\,\ud y\,\Big(\frac{1}{\:|x-y|^{3-s}}-\frac{1}{\:|x|^{3-s}}\Big)^2|V(y)|\\
&\lesssim\; \int_{\R^3}\,\ud y\,|V(y)|\,|y|^{2s-3}\;<\; +\infty\,,
\end{split}
\end{equation*}
as follows from a Cauchy-Schwartz inequality in the second step, from the bound \eqref{eq:moredecay} in the third step, and from the assumption $V\in L^1(\mathbb{R}^3,\langle x\rangle^{2s-3}\ud x)$ in the last step.
 
 Since $|x|^{-(3-s)}\in L^2_\mathrm{loc}(\mathbb{R}^3)$, because $s>\frac{3}{2}$, then identity \eqref{eq:split} implies that $\psi\in L^2_\mathrm{loc}(\mathbb{R}^3)$. Moreover, from \eqref{eq:phi} and \eqref{eq:psi} one finds
 \[
  V\psi\;=\;v u (-\Delta)^{-\frac{s}{2}} v \phi\;=\;-v\phi\;=\;-(-\Delta)^{\frac{s}{2}}\psi\,,
 \]
 whence $((-\Delta)^{\frac{s}{2}}+V)\psi=0$ distributionally. This completes the proof of part (ii).

 Using \eqref{eq:psi} and the distributional identity proved in part (ii) one finds
 \[
  \langle v,\phi\rangle_{L^2}\;=\;\int_{\mathbb{R}^3}\ud x\,v(x)\,\phi(x)\;=\;\int_{\mathbb{R}^3}\ud x\,((-\Delta)^{\frac{s}{2}}\psi)(x)\;=\;-\int_{\mathbb{R}^3}\ud x\,V(x)\,\psi(x)\,,
 \]
 which proves part (iii).

 Last, the identity \eqref{eq:split} also implies that $\psi\in L^2(\mathbb{R}^3)$ is equivalent to $\langle v,\phi\rangle_{L^2}=0$. When this is the case, the identity $((-\Delta)^{\frac{s}{2}}+V)\psi=0$ holds in the $L^2$-sense, implying that $\psi\in\mathcal{D}((-\Delta)^{\frac{s}{2}}+V)$. This completes the proof of part (iv). 
\end{proof}

The proof of Theorem \ref{thm:resonance1D} proceeds along the same line.
%s as above, we only give a sketch of it.

\begin{proof}[Proof of Theorem \ref{thm:resonance1D}]
Part (i) follows from Lemma \ref{lem:compactness1D}(i).

Splitting
 \[\tag{*}
  \begin{split}
   \psi(x)\;&=\;((-\Delta)^{-\frac{s}{2}} v\phi)(x)\;=\;\frac{\,\Lambda_s\langle v,\phi\rangle_{L^2}}{\;|x|^{1-s}}+\psi_1(x)\,,
  \end{split}
 \]
where now $\Lambda_s$ is the constant defined in \eqref{eq:sfGlambda_asympt-2-1D}, and using the assumptions $V\in L^1(\mathbb{R},\langle x\rangle^{2s-1}\ud x)$ we can show that $\psi_1\in L^2(\mathbb{R})$, following the same argument as in the above proof of Theorem \ref{thm:resonance}. Therefore, since $|x|^{-(1-s)}\in L^2_\mathrm{loc}(\mathbb{R})$ because $s>\frac{1}{2}$, from (*) one deduces that $\psi\in L^2_\mathrm{loc}(\mathbb{R})$. Moreover, from \eqref{eq:phi1D} and \eqref{eq:psi1D} one finds $-V\psi\;=\;(-\Delta)^{\frac{s}{2}}\psi$, whence $((-\Delta)^{\frac{s}{2}}+V)\psi=0$ distributionally. Thus, part (ii) is proved.

 Next, \eqref{eq:psi1D} and the distributional identity of part (ii) imply
 \[
  \langle v,\phi\rangle_{L^2}\;=-\int_{\mathbb{R}}\ud x\,V(x)\,\psi(x)\,,
 \]
 which proves part (iii).

 Last, the identity (*) also implies that $\psi\in L^2(\mathbb{R})$ is equivalent to $\langle v,\phi\rangle_{L^2}=0$, in which case $((-\Delta)^{\frac{s}{2}}+V)\psi=0$ in the $L^2$-sense, and this implies that $\psi\in\mathcal{D}((-\Delta)^{\frac{s}{2}}+V)$. Thus, part (iv) is proved.
\end{proof}

Let us address now the question of the existence of a potential $V$ such that the fractional Schr\"{o}dinger operator $(-\Delta)^{s/2} + V$ is zero-energy resonant.

\begin{proposition}\label{prop:exist_res}
Let $\theta\in\mathcal{S}(\R^3)$, with $\theta>0$. Define
\[
\psi\;:=\;\theta*\mathsf{G}_{s,0}\;=\;\theta*\frac{\Lambda_s}{\,|x|^{3-s}}\,,\qquad V\;:=\;-\frac{\theta}{\psi}\,,
\]
where $\Lambda_s$ is the constant defined in \eqref{eq:sfGlambda_asympt-2}.
Then $V$ satisfies part $(\mathrm{i})$ of Assumption $(\mathrm{I}_s)$, and $(-\Delta)^{s/2}+V$ is zero-energy resonant on $L^2(\mathbb{R}^3)$, with zero-energy resonance $\psi$.
\end{proposition}

\begin{proof}
By construction $\psi>0$, being the convolution of two strictly positive functions. Moreover, from 
\[
 \widehat{\psi}(p)\;=\;\frac{\,\widehat{\theta}(p)}{|p|^s}
\]
one sees that $\psi$ is continuous, as $\widehat{\psi}\in L^1(\R^3)$, and that $\psi\notin L^2(\mathbb{R}^3)$, as $\widehat{\psi}\notin L^2(\mathbb{R}^3)$ either. Still, for every compact $K\subset \mathbb{R}^3$
\[
 \|(\theta*\mathsf{G}_{s,0})\mathbf{1}_K\|_{L^2}\;\leqslant\;\|\theta\|_{L^\frac{6}{5}}\|\mathsf{G}_{s,0}\|_{L^2}\|\mathbf{1}_K\|_{L^6}\;<\;+\infty\,,
\]
as follows by means of a Schwartz and a Young inequality. 
Thus,
\[\tag{i}
 \psi\;\in\; L^2_{\mathrm{loc}}(\mathbb{R}^3)\setminus L^2(\mathbb{R}^3)\,.
\]

% We have that
% \begin{equation}\label{scra}
% \widehat{\psi}(p)\;=\;(2\pi)^{\frac{3}{2}}\frac{\widehat{\theta}(p)}{|p|^s}\;=\;(2\pi)^{\frac{3}{2}}\frac{\,\widehat{\theta}(0)}{|p|^s}+(2\pi)^{\frac{3}{2}}\frac{\,\widehat{\theta}(p)-\widehat{\theta}(0)}{|p|^s}.
% \end{equation}
% Since $\widehat{\theta}(p)|p|^{-s}\in L^1(\R^3)$, we get that $\psi$ is a continuous function.
% Moreover, it follows from \eqref{scra} that
% $$\psi(x)=\frac{(2\pi)^{\frac32}\Lambda_s\widehat{\phi}(0)}{|x|^{3-s}}+\psi_1(x),$$
% where 
% $$\psi_1(x):=\mathcal{F}^{-1}\Big(\,\frac{\widehat{\theta}(p)-\widehat{\theta}(0)}{|p|^s}\,\Big)(x).$$
% Owing to the bound $|\widehat{\theta}(p)-\widehat{\theta}(0)|\lesssim |p|$, valid for small $p$, we get that $\psi_1\in L^2(\R^3)$. Moreover, since $\widehat{\theta}(0)\neq 0$, 
% 
% \textcolor{blue}{[commented in the LaTeX code] ...we deduce that $\psi\in  L^2_\mathrm{loc}(\mathbb{R}^3)\setminus  L^2(\mathbb{R}^3)$ and that $\psi$ decay at infinity as $|x|^{s-3}$.}

Next, we argue that the leading decay in $\psi$ is $|x|^{-(3-s)}$. To see that, since $\theta(x)\lesssim\langle x\rangle^{-m}$ for any $m\in\mathbb{N}$, we write
\[\tag{ii}
 \begin{split}
  \psi(x)\;&\lesssim\;\int_{\mathbb{R}^3}\frac{1}{\langle x-y\rangle^m}\,\frac{1}{\:|y|^{3-s}}\,\ud y \\
  &=\;\int_{|y-x|\geqslant\frac{1}{2}|x|}\frac{1}{\langle y-x\rangle^m}\,\frac{1}{\:|y|^{3-s}}\,\ud y+\int_{|y-x|<\frac{1}{2}|x|}\frac{1}{\langle y-x\rangle^m}\,\frac{1}{\:|y|^{3-s}}\,\ud y
%   \\
%   &=\;\int_{|y-x|\geqslant\frac{1}{2}|x|}\frac{1}{\langle x-y\rangle^m}\,\frac{1}{\:|y|^{3-s}}\,\ud y+\frac{1}{\:|x|^{3-s}}\int_{|y-x|<\frac{1}{2}|x|}\frac{\ud y}{\langle y-x\rangle^m}\,.
 \end{split}
\]
and we take non-restrictively $|x|\geqslant 1$ and $m\geqslant 4$. The first integral in the r.h.s.~of (ii) is estimated as
\[
 \begin{split}
  &\int_{|y-x|\geqslant\frac{1}{2}|x|}\frac{1}{\langle y-x\rangle^m}\,\frac{1}{\:|y|^{3-s}}\,\ud y \\
  &\quad=\;\int_{\substack{|y-x|\geqslant\frac{1}{2}|x| \\ |y|\geqslant \frac{3}{2}|x|\geqslant\frac{3}{2}}}\frac{1}{\langle y-x\rangle^m}\,\frac{1}{\:|y|^{3-s}}\,\ud y+\int_{\substack{|y-x|\geqslant\frac{1}{2}|x| \\ |y|< \frac{3}{2}|x|}}\frac{1}{\langle y-x\rangle^m}\,\frac{1}{\:|y|^{3-s}}\,\ud y \\
  &\quad\lesssim\;\frac{1}{\:|x|^{m-4}}\int_{\mathbb{R}^3}\frac{\ud y}{\langle y-x\rangle^4}+\frac{1}{\:|x|^m}\int_{|y|<\frac{3}{2}|x|}\frac{\ud y}{\:|y|^{3-s}}\;\lesssim\;\frac{1}{\:|x|^{m-4}}+\frac{1}{\:|x|^{m-s}}
 \end{split}
\]
and hence vanishes as $|x|\to +\infty$ faster than any power. Concerning the second integral in the r.h.s.~of (ii), since $|y-x|<\frac{1}{2}|x|$ implies $|y|>\frac{1}{2}|x|$, one has
% \[
%  \begin{split}
%   &|x|^{3-s}\!\int_{|y-x|<\frac{1}{2}|x|}\frac{1}{\langle y-x\rangle^m}\,\frac{1}{\:|y|^{3-s}}\,\ud y\;=\;\int_{|y-x|<\frac{1}{2}|x|}\frac{1}{\langle y-x\rangle^m}\,\Big(\frac{|x|}{|y|}\Big)^{3-s}\,\ud y
%  \end{split}
% \]
% 
% 
\[
\int_{|y-x|<\frac{1}{2}|x|}\frac{1}{\langle y-x\rangle^m}\,\frac{1}{\:|y|^{3-s}}\,\ud y\;\lesssim\;\frac{1}{\:|x|^{3-s}}\int_{\mathbb{R}^3}\frac{\ud y}{\langle y-x\rangle^m}\;\lesssim\;\frac{1}{\:|x|^{3-s}}\,.
\]
Therefore, (ii) implies
\[\tag{iii}
 \psi(x)\;\lesssim\;\langle x\rangle^{-(3-s)}\,.
\]

Now, splitting
\[
 \psi\;=\;\frac{\Lambda_s\,\widehat{\theta}(0)}{|x|^{3-s}}+\psi_1\,,\qquad \psi_1\;=\;\Big(\,\frac{\widehat{\theta}(\cdot)-\widehat{\theta}(0)}{|\cdot|^s}\,\Big)^{\!\vee}
\]
and observing that $\widehat{\psi_1}\in L^2(\mathbb{R}^3)$ (owing to the bound $|\widehat{\theta}(p)-\widehat{\theta}(0)|\lesssim |p|$, valid for small $|p|$), whence also $\psi_1\in L^2(\mathbb{R}^3)$, we deduce that in order for the decay $\psi_1(x)\lesssim\langle x\rangle^{-(3-s)}$ (whose bound is \emph{not} square-integrable) implied by (iii) above to be compatible with the square-integrability of the continuous function $\psi_1$, $\psi_1$ itself must decay more that $\langle x\rangle^{-(3-s)}$, which allows us to conclude that the leading decay in $\psi$ is $|x|^{-(3-s)}$.

Since $\psi$ is continuous, positive, and with polynomial decay $|x|^{-(3-s)}$ at infinity, then $1/\psi$ is continuous, positive, and with polynomial growth at infinity. 
Since $\theta$ is a Schwartz function, we conclude that $V=\theta\frac{1}{\psi}$ is continuous and decays at infinity faster than any polynomial.

As a consequence, $V\in L^1(\mathbb{R}^3,\langle x\rangle^{2s-3}\ud x)$ and also $\|V\|_{\mathcal{R}_{s,3}}\lesssim\|V\|_{L^{3/s}}<+\infty$. This proves that $V$ satisfies part (i) of Assumption (I$_s$).

By construction, $(-\Delta)^{s/2}\psi=\theta$, whence
\[\tag{iv}
 ((-\Delta)^{s/2}+V)\psi\;=\;0
\]
distributionally. This also implies
\begin{equation*}\tag{v}
(-\Delta)^{-s/2}V\psi\;=\;-\psi\,.
\end{equation*}

Writing as usual $V=uv$ with $v=|V|^{\frac{1}{2}}$ and $u=|V|^{\frac{1}{2}}\mathrm{sign}(V)$, let us now set $\phi:=-u\psi$. Then (v) yields
\[\tag{vi}
 \psi\;=\;(-\Delta)^{-\frac{s}{2}}v\phi\,,
\]
which is precisely the relation between $\psi$ and $\phi$ given by \eqref{eq:psi}. $\phi$ cannot be identically zero, because so would be $\psi$, owing to (vi), which is not the case. Moreover, $\phi$ is square-integrable, because in the product $u\psi$ the $L^2_{\mathrm{loc}}$-function $\psi$ is multiplied by a function that decays more than polynomially at infinity. Therefore,
\[\tag{vii}
 \phi\;\in\; L^2(\R^3)\setminus\{0\}\,,
\]
and, multiplying (v) by $u$,
\begin{equation*}\tag{viii}
u(-\Delta)^{-\frac{s}{2}}v\phi\;=\;-\phi
\end{equation*}
as an identity in $L^2(\mathbb{R}^3)$.

Conditions (i), (iv), (vi), (vii), and (viii) above, owing to Theorem \ref{thm:resonance} and to the present definition of resonance, imply that $(-\Delta)^{s/2}+V$ is zero-energy resonant, with zero-energy resonance $\psi$.
\end{proof}

The counterpart result in 1D, which we state here below, has a completely analogous proof, that we then omit.

\begin{proposition}\label{prop:exist_res_1D}
Let $\theta\in\mathcal{S}(\R)$, with $\theta>0$. Define
\[
\psi\;:=\;\theta*\mathsf{G}_{s,0}\;=\;\theta*\frac{\Lambda_s}{\,|x|^{1-s}}\,,\qquad V\;:=\;-\frac{\theta}{\psi}\,,
\]
where $\Lambda_s$ is the constant defined in \eqref{eq:sfGlambda_asympt-2-1D}
Then $V$ satisfies part $(\mathrm{i})$ of Assumption $(\mathrm{I}^-_s)$, and $(-\Delta)^{s/2}+V$ is zero-energy resonant on $L^2(\mathbb{R})$, with zero-energy resonance $\psi$.
\end{proposition}

\begin{remark}
 By means of a more refined discussion, in the same spirit of \cite{Klaus-Simon-1980-resonances}, we can identify the threshold coupling parameter $\lambda\in\mathbb{R}$, for a given potential $V$ in a suitable class, for which $(-\Delta)^{s/2}+\lambda V$ is zero-energy resonant. We do not develop this interesting approach here.
\end{remark}

\def\cprime{$'$}


\begin{thebibliography}{10}

\bibitem{AHKH-1987-2D}
{\sc S.~Albeverio, F.~Gesztesy, R.~{H{\o} egh-Krohn}, and H.~Holden}, {\em
  {Point interactions in two dimensions: basic properties, approximations and
  applications to solid state physics}}, J. Reine Angew. Math., 380 (1987),
  pp.~87--107.

\bibitem{AGHK-1982}
{\sc S.~Albeverio, F.~Gesztesy, and R.~H{\o}egh-Krohn}, {\em {The low energy
  expansion in nonrelativistic scattering theory}}, Ann. Inst. H. Poincar{\'e}
  Sect. A (N.S.), 37 (1982), pp.~1--28.

\bibitem{albeverio-solvable}
{\sc S.~Albeverio, F.~Gesztesy, R.~H{\o}egh-Krohn, and H.~Holden}, {\em
  {Solvable {M}odels in {Q}uantum {M}echanics}}, {Texts and Monographs in
  Physics}, Springer-Verlag, New York, 1988.

\bibitem{aghkk-1984-1D-point-int}
{\sc S.~Albeverio, F.~Gesztesy, R.~H{\o}egh-Krohn, and W.~Kirsch}, {\em {On
  point interactions in one dimension}}, J. Operator Theory, 12 (1984),
  pp.~101--126.

\bibitem{albverio-kurasov-2000-sing_pert_diff_ops}
{\sc S.~Albeverio and P.~Kurasov}, {\em {Singular perturbations of differential
  operators}}, vol.~271 of {London Mathematical Society Lecture Note Series},
  Cambridge University Press, Cambridge, 2000.
\newblock Solvable Schr{\"o}dinger type operators.

\bibitem{COliveira-Vaz-JPA2011_1D_fracLaplandDelta}
{\sc E.~{Capelas de Oliveira} and J.~J. Vaz}, {\em {Tunneling in fractional
  quantum mechanics}}, J. Phys. A, 44 (2011), pp.~185303, 17.

\bibitem{DM-2015-halfline}
{\sc G.~Dell'Antonio and A.~Michelangeli}, {\em {Schr{\"o}dinger operators on
  half-line with shrinking potentials at the origin}}, Asymptot. Anal., 97
  (2016), pp.~113--138.

\bibitem{GMO-KVB2017}
{\sc M.~Gallone, A.~Michelangeli, and A.~Ottolini}, {\em
  {Kre{\u\i}n-Vi\v{s}ik-Birman self-adjoint extension theory revisited}}, SISSA
  preprint 25/2017/MATE (2017).

\bibitem{Jarosz-Vaz-JMP2016_1D_gs_fracLaplandDelta}
{\sc S.~Jarosz and J.~J. Vaz}, {\em {Fractional {S}chr{\"o}dinger equation with
  {R}iesz-{F}eller derivative for delta potentials}}, J. Math. Phys., 57
  (2016), pp.~123506, 16.

\bibitem{Kato-perturbation}
{\sc T.~Kato}, {\em {Perturbation theory for linear operators}}, {Classics in
  Mathematics}, Springer-Verlag, Berlin, 1995.
\newblock Reprint of the 1980 edition.

\bibitem{Klaus-Simon-1980-resonances}
{\sc M.~Klaus and B.~Simon}, {\em {Coupling constant thresholds in
  nonrelativistic quantum mechanics. {I}. {S}hort-range two-body case}}, Ann.
  Physics, 130 (1980), pp.~251--281.

\bibitem{konno-kuroda-1966}
{\sc R.~Konno and S.~T. Kuroda}, {\em {On the finiteness of perturbed
  eigenvalues}}, J. Fac. Sci. Univ. Tokyo Sect. I, 13 (1966), pp.~55--63
  (1966).

\bibitem{Lenzi-etAl-JMP2013_fracLapl_andDelta}
{\sc E.~K. Lenzi, H.~V. Ribeiro, M.~A.~F. dos Santos, R.~Rossato, and R.~S.
  Mendes}, {\em {Time dependent solutions for a fractional {S}chr{\"o}dinger
  equation with delta potentials}}, J. Math. Phys., 54 (2013), pp.~082107, 8.

\bibitem{MOS-2018-fractional-and-perturbation}
{\sc A.~Michelangeli, A.~Ottolini, and R.~Scandone}, {\em {Fractional powers
  and singular perturbations of differential operators}}, arXiv:1801.08885
  (2018).

\bibitem{Muslih-IntJTP-2010-3D-fracLaplandDelta}
{\sc S.~I. Muslih}, {\em {Solutions of a particle with fractional
  {$\delta$}-potential in a fractional dimensional space}}, Internat. J.
  Theoret. Phys., 49 (2010), pp.~2095--2104.

\bibitem{Nayga-Esguerra-JMP-2016}
{\sc M.~M. Nayga and J.~P. Esguerra}, {\em {Green's functions and energy
  eigenvalues for delta-perturbed space-fractional quantum systems}}, J. Math.
  Phys., 57 (2016), pp.~022103, 7.

\bibitem{COliveira-Costa-Vaz-JMP2010}
{\sc E.~C.~d. Oliveira, F.~S. Costa, and J.~J. Vaz}, {\em {The fractional
  {S}chr{\"o}dinger equation for delta potentials}}, J. Math. Phys., 51 (2010),
  pp.~123517, 16.

\bibitem{Sacchetti-fractional2018}
{\sc A.~Sacchetti}, {\em {Stationary solutions of a fractional Laplacian with singular perturbation}}, arXiv:1801.01694 (2018).

\bibitem{Sandev-Petreska-Lenzi-JMP2014}
{\sc T.~Sandev, I.~Petreska, and E.~K. Lenzi}, {\em {Time-dependent
  {S}chr{\"o}dinger-like equation with nonlocal term}}, J. Math. Phys., 55
  (2014), pp.~092105, 10.

\bibitem{simon_trace_ideals}
{\sc B.~Simon}, {\em {Trace ideals and their applications}}, vol.~120 of
  {Mathematical Surveys and Monographs}, American Mathematical Society,
  Providence, RI, second~ed., 2005.

\bibitem{Tare-Esguerra-JMP2014}
{\sc J.~D. Tare and J.~P.~H. Esguerra}, {\em {Bound states for multiple
  {D}irac-{$\delta$} wells in space-fractional quantum mechanics}}, J. Math.
  Phys., 55 (2014), pp.~012106, 10.

\end{thebibliography}
\end{document}